\def\DateTime{21/February/2014
}
\def\Version{Version $3.0$}
\def\yes{\if00}
\def\no{\if01}
\def\iftenpt{\no}
\def\ifelevenpt{\no}
\def\iftwelvept{\yes}

\def\ifquery{\yes}

\iftenpt
\documentclass[leqno]{amsart}
\else\fi
\ifelevenpt
\documentclass[leqno,11pt]{amsart}
\else\fi
\iftwelvept
\documentclass[leqno,12pt]{amsart}
\else\fi

\usepackage{amssymb}
\usepackage{amscd}
\usepackage{verbatim,color}
\usepackage{mathrsfs,enumerate}
\usepackage[all]{xy}

\usepackage[sc]{mathpazo} 
\usepackage[T1]{fontenc}

\ifelevenpt
\else\fi

\iftwelvept
\else\fi

\theoremstyle{plain}
\newtheorem{Theorem}{Theorem}[section]
\newtheorem{Proposition}[Theorem]{Proposition}
\newtheorem{Lemma}[Theorem]{Lemma}

\newtheorem{Corollary}[Theorem]{Corollary}

\newtheorem{Claim}{Claim}[Theorem]

\theoremstyle{definition}
\newtheorem{Definition}[Theorem]{Definition}
\newtheorem{Remark}[Theorem]{Remark}
\newtheorem{Example}[Theorem]{Example}

\newtheorem{Question}[Theorem]{Question}

\renewcommand{\theTheorem}{\arabic{section}.\arabic{Theorem}}
\renewcommand{\theClaim}{\arabic{section}.\arabic{Theorem}.\arabic{Claim}}
\numberwithin{equation}{section}

\def\rom{\textup}
\newcommand{\ZZ}{{\mathbb{Z}}}
\newcommand{\QQ}{{\mathbb{Q}}}
\newcommand{\RR}{{\mathbb{R}}}

\newcommand{\CC}{{\mathbb{C}}}
\newcommand{\PP}{{\mathbb{P}}}

\newcommand{\FF}{{\mathbb{F}}}

\newcommand{\AAA}{{\mathscr{A}}}
\newcommand{\CCC}{{\mathscr{C}}}
\newcommand{\DDD}{{\mathscr{D}}}
\newcommand{\EEE}{{\mathscr{E}}}

\newcommand{\LLL}{{\mathscr{L}}}
\newcommand{\XXX}{{\mathscr{X}}}
\newcommand{\Proj}{\operatorname{Proj}}

\newcommand{\Rat}{\operatorname{Rat}}

\newcommand{\Spec}{\operatorname{Spec}}

\newcommand{\Gal}{\operatorname{Gal}}
\newcommand{\Supp}{\operatorname{Supp}}

\newcommand{\ord}{\operatorname{ord}}

\newcommand{\adeg}{\widehat{\operatorname{deg}}}

\newcommand{\Div}{\operatorname{Div}}
\newcommand{\aDiv}{\widehat{\operatorname{Div}}}

\newcommand{\vol}{\operatorname{vol}}
\newcommand{\avol}{\widehat{\operatorname{vol}}}

\newcommand{\an}{\operatorname{an}}
\newcommand{\Prep}{\operatorname{Prep}}
\newcommand{\Tpsh}{\operatorname{PSH}}

\newcommand{\rest}[2]{\left.{#1}\right\vert_{{#2}}}
\newcommand{\srest}[2]{{#1}\vert_{{#2}}}
\ifquery
\def\query#1{\setlength\marginparwidth{65pt} 
\marginpar{\raggedright\fontsize{7.81}{9} 
\selectfont\upshape\hrule\smallskip 
#1\par\smallskip\hrule}} 

\else
\def\query#1{}
\fi
\newcommand{\ad}{\operatorname{a}}

\begin{document}

\title[Algebraic dynamical systems and Dirichlet's unit theorem]%
{Algebraic dynamical systems and \\ Dirichlet's unit theorem on arithmetic varieties}
\author{Huayi Chen}
\address{Universit\'e Grenoble Alpes, Institut Fourier (UMR 5582), F-38402 Grenoble, France}
\email{huayi.chen@ujf-grenoble.fr}
\author{Atsushi Moriwaki}
\address{Department of Mathematics, Faculty of Science,
Kyoto University, Kyoto, 606-8502, Japan}
\email{moriwaki@math.kyoto-u.ac.jp}
\date{\DateTime, (\Version)}
\begin{abstract}
{In this paper,  we study obstructions to the Dirichlet property by two approaches~: density of non-positive points and  functionals on adelic $\mathbb R$-divisors. Applied to the algebraic dynamical systems, these results provide examples of nef adelic arithmetic $\RR$-Cartier divisor 
which does not have the Dirichlet property. We hope the obstructions obtained in the article will give ways toward criteria of the Dirichlet property.
} 
\end{abstract}


\maketitle

\section*{Introduction}
Let $X$ be a projective geometrically integral variety over a number field $K$ and
let $\overline{D} = (D, g)$ be an adelic arithmetic $\RR$-Cartier divisor of $C^0$-type on $X$
(for details of adelic arithmetic divisors, see \cite{MoAdel}).
We say that $\overline{D}$ has {\em the Dirichlet property} if
$\overline{D} + \widehat{(\psi)}$ is effective for some $\psi \in \Rat(X)^{\times}_{\RR} (:= \Rat(X)^{\times} \otimes_{\ZZ} \RR)$. It is clear that, if $\overline D$ has the Dirichlet property, then it is pseudo-effective, namely for any big adelic arithmetic $\mathbb R$-Cartier divisor $\overline E$, the sum $\overline D+\overline E$ is also big.
In \cite{MoD}, the following question has been proposed:

\begin{quote}
If $\overline{D}$ is pseudo-effective, 
does it follow that
$\overline{D}$ has the Dirichlet property?
\end{quote}
In the case where $X=\Spec K$, the pseudo-effectivity actually implies the Dirichlet property. It can be considered as an Arakelov geometry interpretation of the classical Dirichlet unit theorem. Therefore the above problem could be seen as the study of possible higher dimensional generalizations of the Dirichlet unit theorem.

It is know that the above question has a positive answer in the following cases:

\begin{enumerate}
\renewcommand{\labelenumi}{(\arabic{enumi})}
\item
$X = \Spec(K)$ (the classical Dirichlet unit theorem).

\item
$D$ is numerically trivial on $X$ (cf. \cite{MoD,MoAdel}).

\item
$X$ is a toric variety and $\overline{D}$ is of toric type (cf. \cite{BMPS}).
\end{enumerate}
The purpose of this paper is to give a negative answer to the above question and to study the obstructions to the Dirichlet property. We will construct from an algebraic dynamical system over a number field a nef adelic arithmetic Cartier divisor $\overline{D}$ 
which does not have the Dirichlet property.
The obstruction comes from the denseness of the set of preperiodic points with respect to the analytic topology.
More precisely,
let $f : X \to X$ be a surjective endomorphism of $X$ over $K$.
Let $D$ be an ample $\RR$-Cartier divisor on $X$ such that $f^*(D) = dD + (\varphi)$
for some real number $d > 1$ and $\varphi \in \Rat(X)^{\times}_{\RR}$.
Let $\overline{D} = (D, g)$ be the canonical compactification of $D$, that is,
$\overline{D}$ is an adelic arithmetic $\RR$-Cartier divisor of $C^0$-type 
with $f^*(\overline{D}) = d\overline{D} + \widehat{(\varphi)}$ (cf. Section~\ref{sec:canonical:comp}).
Note that $\overline{D}$ is nef (cf. Lemma~\ref{lem:can:comp:nef}). 
The main result of this paper is the following:

\begin{Theorem}[cf. Theorem~\ref{thm:canonical:comp}]
\label{thm:main:result}
If the set of all preperiodic points of $f$ 
is dense on some connected component of  $X(\CC)$  with respect to
the analytic topology \rom{(}the topology as an analytic space\rom{)}, then 
the Dirichlet property of $\overline{D}$ does not hold.
\end{Theorem}

The proof of the theorem relies on a necessary condition of the Dirichlet property established in Lemma \ref{lem:non:dense}. We actually prove that the essential support (see \eqref{Equ:esssupport} for definition) of algebraic points with non-positive heights should not meet the strictly effective locus of an effective section of the adelic arithmetic $\mathbb R$-Cartier divisor.

The concrete examples to apply the above theorem  are discussed in Section~\ref{sec:examples}. 
Even for the algebraic dynamical system as treated in Theorem~\ref{thm:main:result},
it is a very interesting and challenging problem to find a non-trivial sufficient condition to
ensure the Dirichlet property.
Further, in \cite{MoGD}, we introduce a geometric analogue of the above question.
Namely, if $D$ is a pseudo-effective $\QQ$-Cartier divisor on a normal projective variety
defined over a finite field, can we conclude that $D$ is $\QQ$-effective?
It actually holds on a certain kind of an abelian scheme over a curve,
so that the situation of the geometric case is
slightly different from the arithmetic case.

Note that the essential support of a family $S$ of algebraic points is not empty only if the family $S$ is Zariski dense. Therefore, the lemma \ref{lem:non:dense} provides non-trivial necessary conditions for the Dirichlet property only when the set of non-positive points is Zariski dense. In order to treat general adelic arithmetic $\mathbb R$-Cartier divisors, we propose the functional approach. We introduce the notion of asymptotic maximal slope for any adelic arithmetic $\mathbb R$-Cartier divisor on $X$ (see \S\ref{Subsec:asypmaxs} and \S\ref{Sec:extension}), which is the threshold of the Dirichlet property where we consider the twists of the adelic arithmetic $\mathbb R$-Cartier divisor by the pull-back of adelic arithmetic $\mathbb R$-Cartier divisors on $\Spec K$. We prove that this arithmetic invariant also determines the pseudo-effectivity of the adelic arithmetic $\mathbb R$-Cartier divisor (see Proposition \ref{Pro:cripseudoeff}). Therefore the Dirichlet property and the pseudo-effectivity are naturally linked by this numerical invariant. We then obtain a necessary condition of the Dirichlet property in terms of the directional derivative of the asymptotic maximal slope, which is a functional on the space of all adelic arithmetic $\mathbb R$-Cartier divisors (non-necessarily additive \emph{a priori}). For this purpose we establish a general analysis for functionals on the spaces of adelic arithmetic $\mathbb R$-Cartier divisors as in Theorem \ref{Thm:consequenceDirichlet}. This result can be applied to not only the maximal asymptotic slope (see Corollary \ref{Cor:muasy}) but also other natural arithmetic invariants such as the arithmetic volume function (see Corollary \ref{Cor:volder}) and the arithmetic self-intersection number (see Corollary \ref{Cor:suppD}). In \S\ref{Subsec:comparison}, we compare these specifications of Theorem \ref{Thm:consequenceDirichlet}. The comparisons show that the arithmetic maximal slope is particularly adequate in the study of the Dirichlet property of pseudo-effective adelic arithmetic $\mathbb R$-Cartier divisors.

We conclude the article by a refined version of the question above (see Question \ref{Que:refined}). We hope that our work will provide clues for the further research on criteria of the Dirichlet property for higher dimension arithmetic variety.

\medskip
Finally
we would like to express thanks to Prof. Burgos i Gil, 
Prof. Kawaguchi and Prof. Yuan for their valuable and helpful comments.

\renewcommand{\thesubsubsection}{\arabic{subsubsection}}
\subsection*{Conventions and terminology}
In this paper,
we frequently use the same notations as in \cite{MoArZariski} and \cite{MoAdel}.

\subsubsection{}
\label{CV:residue:field}
Let $V$ be a variety over a field $F$ and
$\overline{F}$ an algebraic closure of $F$.
Let $x$ be an $\overline{F}$-valued point of $V$, that is,
a morphism 
\[
x : \Spec(\overline{F}) \to V
\]
over $F$.
The residue field of $V$ at the closed point given by the image of $x$ is denoted by
$F(x)$.

\medskip
In the following,
let $X$ be a projective 
and geometrically integral scheme over a number field $K$. Let $d$ be its Krull dimension.

\subsubsection{}
\label{CV:orbit}
Let $O_K$ be the ring of integers in $K$ and $M_K$ the set of all maximal ideals of $O_K$.
Let $K(\CC)$ be the set of all embeddings $K \hookrightarrow \CC$.
For each $v \in M_K \cup K(\CC)$, we define $K_{v}$ to be
\[
K_{v} := 
\begin{cases}
\text{$K \otimes^{\sigma}_K \CC$ with respect to $\sigma$} & \text{if $v = \sigma \in K(\CC)$},\\
\text{the completion of $K$ at $\mathfrak p$} & \text{if $v = \mathfrak p \in M_K$}.
\end{cases}
\]
Moreover, let $X_{v}$ denote the fiber product $X \times_{\Spec(K)} \Spec(K_{v})$.
Note that $K_{\sigma}$ is naturally isomorphic to $\CC$
via $a \otimes^{\sigma} z \mapsto \sigma(a)z$ and $X_{\sigma}$ is nothing more than
the fiber product $X \times^{\sigma}_{\Spec(K)} \Spec(\CC)$ 
with respect to $\sigma$.

Let $x$ be a $\overline{K}$-valued point of $X$.
Let $\{ \phi_1, \ldots, \phi_n \}$ be the set of all $K_{v}$-algebra homomorphisms 
$K(x) \otimes_K K_{v} \to \overline{K_{v}}$,
where $\overline{K_{v}}$ is an algebraic closure of $K_{v}$.
Note that $n = [K(x) : K]$.
For each $i=1, \ldots, n$, let $w_i$ be the $\overline{K_{v}}$-valued point
of $X_{v}$  
given by the composition of morphisms
\[
\begin{CD}
\Spec(\overline{K_{v}}) @>{\phi_i^a}>> \Spec(K(x) \otimes_K K_{v}) 
@>{x \times \operatorname{id}_{K_{v}}}>> X_{v},
\end{CD}
\]
where $\phi_i^a$ is the morphism of $K_{v}$-schemes induced by $\phi_i$.
We denote $\{ w_1, \ldots, w_n \}$ by $O_{v}(x)$. 

\subsubsection{}
\label{CV:analytication}
For $v \in M_K \cup K(\CC)$,
the analytification $X_{v}^{\an}$ of $X_{v}$ is defined by
\[
X_{v}^{\an} :=
\begin{cases}
X_{\sigma}(K_{\sigma}) & \text{if $v = \sigma \in K(\CC)$}, \\
\text{$X^{\an}_{\mathfrak p}$ in the sense of Berkovich \cite{Be}}
& \text{if $v = \mathfrak p \in M_K$}.
\end{cases}
\]
As $K_{\sigma}$ is naturally identified with $\CC$, 
$X^{\an}_{\sigma} = X_{\sigma}(K_{\sigma}) = X_{\sigma}(\CC)$.
We equip the space $X_{\sigma}^{\an}$ (resp. $X^{\an}_{\mathfrak p}$) with the analytic topology, namely the topology as an analytic space
(resp. as a Berkovich space).
Let $X_{\infty}^{\an}$ denote  the set of $\CC$-valued points of $X$ over $\QQ$.
Note that $X^{\an}_{\infty} = \coprod_{\sigma \in K(\CC)} X_{\sigma}^{\an}$.
We often denote $X^{\an}_{\infty}$ by $X(\CC)$. Note that the complex conjugation induces an involution $F_\infty:X(\mathbb C)\rightarrow X(\mathbb C)$.

\subsubsection{}
\label{CV:analytic:point}
Let us fix $v \in M_K \cup K(\CC)$. For a $\overline{K_{v}}$-valued point $w$ of $X_{v}$,
we define $w^{\an} \in X_{v}^{\an}$ to be
\[
w^{\an} :=
\begin{cases}
w & \text{if $v = \sigma \in K(\CC)$}, \\
\text{the valuation of $K_{\mathfrak p}(w)$
as an extension of $v_{\mathfrak p}$}
& \text{if $v = \mathfrak p \in M_K$},
\end{cases}
\]
where $v_{\mathfrak p}$ is the valuation of $K_{\mathfrak p}$ defined as $v_{\mathfrak p}(f)=\#(O_K/{\mathfrak p})^{-\ord_{\mathfrak p}(f)}$.
Note that 
\[
\#\{ w' \in X_{v}(\overline{K_{v}}) \mid {w'}^{\an} = w^{\an} \} = [K_{v}(w) : K_{v}].
\]

\subsubsection{}\label{CV:adelic:arith:div}
Let $\Div(X)$ be the group of Cartier divisors on $X$ and denote by $\Div_{\mathbb R}(X)$ the $\mathbb R$-vector space $\Div(X)\otimes_{\mathbb Z}\mathbb R$. If $\varphi$ is an element in $\Rat(X)^{\times}$, we denote by $(\varphi)$ its divisor, which is an element in $\Div(X)$. The Cartier divisors on $X$ constructed in this way are called \emph{principal divisors}. The map $\varphi\mapsto (\varphi)$ is a group homomorphism and extends by extension of scalar to an $\mathbb R$-linear map $\Rat(X)^{\times}_{\mathbb R}\rightarrow\Div_{\mathbb R}(X)$, where $\Rat(X)^{\times}_{\RR} := \Rat(X)^{\times} \otimes_{\ZZ} \RR$.

By an \emph{adelic arithmetic $\mathbb R$-Cartier divisor} of $C^0$-type on $X$, we refer to a pair $\overline D=(D,g)$, where $D\in\Div_{\mathbb R}(X)$ and $g=(g_{v})_{v\in M_K\cup K(\mathbb C)}$ is a family of Green functions, with $g_{v}$ being a $D$-Green function of $C^0$-type on $X_{v}^{\mathrm{an}}$. We also require that $g_{\mathfrak p}$ comes from an integral model of $D$ for all but a finite number of $\mathfrak p\in M_K$, and that the family $(g_\sigma)_{\sigma\in K(\mathbb C)}$ is invariant under the action of $F_\infty$. The family $g = (g_{v})_{v \in M_K \cup K(\CC)}$ is often denoted by
\[
\sum_{v \in M_K \cup K(\CC)} g_{v} [v].
\] 
If it is not specified, an adelic arithmetic $\mathbb R$-Cartier divisor refers to an adelic arithmetic $\mathbb R$-Cartier divisor of $C^0$-type. We denote the vector space consisting of adelic arithmetic $\RR$-Cartier divisors on $X$ by
$\aDiv_{\RR}(X)$. If $\varphi$ is an element in $\Rat(X)^{\times}$, we define an adelic arithmetic $\mathbb R$-Cartier divisor as follows
\[\widehat{(\varphi)}:=\Big((\varphi),\sum_{v\in M_K\cup K(\mathbb C)} -\log|\varphi_v|^2_{v}\;[v] \Big),\]
where $\varphi_{v}$ is the rational function on $X_{v}^{\mathrm{an}}$ induced by $\varphi$. The map $\Rat(X)^{\times}\rightarrow\widehat{\Div}_{\mathbb R}(X)$ extends naturally to $\Rat(X)^{\times}_{\mathbb R}$ and defines an $\mathbb R$-linear homomorphism of vector spaces. Any element in the image of this $\mathbb R$-linear map is called a \emph{principal} adelic arithmetic $\mathbb R$-Cartier divisor. 

For any $v \in M_K\cup K(\mathbb C)$, one has a natural embedding from the vector space $C^0(X_{v}^{\mathrm{an}})$ into $\widehat{\mathrm{Div}}_{\mathbb R}(X)$ which sends $f_{v}\in C^0(X_{v}^{\mathrm{an}})$ to 
\[
\begin{cases}
\Big(  0, f_{v} [v] \Big) & \text{if $v \in M_K$}, \\[2ex]
\Big(  0, \frac{1}{2} f_{v} [v] + \frac{1}{2} F_{\infty}^*(f_{v}) [\bar{v}] \Big) & \text{if $v \in K(\CC)$}.
\end{cases}
\]
We denote by $\overline O(f_{v})$ this adelic arithmetic $\mathbb R$-Cartier divisor.

In the particular case where $X=\Spec K$, an adelic arithmetic $\mathbb R$-Cartier divisor $\zeta$ on $\Spec K$ is a vector $(\zeta_{v})_{v\in M_K\cup K(\mathbb C)}$ in $\mathbb R^{M_K\cup K(\mathbb C)}:=\mathrm{Map}(M_K\cup K(\mathbb C),\mathbb R)$ which we can write into the form of a formal sum \[\sum_{v\in M_K\cup K(\mathbb C)}\zeta_{v}[v],\]
where $\zeta_{v}=0$ for all but a finite number of indices $v$.
The \emph{Arakelov degree} of $\zeta$ is then defined as
\[\widehat{\deg}(\zeta)=\frac 12\sum_{v\in M_K\cup K(\mathbb C)}\zeta_{v}.\]
For more details of $\adeg(.)$, see \cite[SubSection~4.2]{MoAdel}

\subsubsection{}
\label{CV:height}
Let $\overline{D} = (D, g)$ be an adelic arithmetic $\RR$-Cartier divisor of $C^0$-type on $X$.
For any algebraic point $x$ of $X$ outside the support of $D$, the \emph{normalized height} $h_{\overline{D}}(x)$ of $x$ with respect to $\overline{D}$ is defined to be
\[
h_{\overline{D}}(x) := \frac{\adeg(\srest{\overline{D}}{x})}{[K(x): K]}
= \frac{1}{[K(x):K]} \sum_{v \in M_K \cup K(\CC)} \sum_{w \in O_{v}(x)} 
\frac{1}{2} \ g_{v}(w^{\an}).
\]
This function can be extended to the set of all points in $X(\overline K)$, see \cite[\S4.2]{MoAdel} for details.
A $\overline{K}$-valued point $x$ of $X$ is said to be {\em non-positive with respect to $\overline{D}$}
if $h_{\overline{D}}(x) \leq 0$. Note that the height function $h_{\overline D}(.)$ does not change if we replace $\overline D$ by $\overline D+\widehat{(\phi)}$ with $\phi\in\Rat(X)_{\mathbb R}^{\times}$. This is a consequence of the product formula for the number field $K$.

For any real number $\lambda$, we denote by $X(\overline{K})^{\overline{D}}_{\leq \lambda}$  the set of all
$\overline{K}$-valued points of $X$ whose height with respect to $\overline D$ is bounded from above by $\lambda$, namely
\[
X(\overline{K})^{\overline{D}}_{\leq \lambda} := \{ x \in X(\overline{K}) \mid h_{\overline{D}}(x) \leq \lambda \}.
\]
The \emph{essential minimum} of the height function $h_{\overline D}(\cdot)$ is defined as 
\[\widehat{\mu}_{\mathrm{ess}}(\overline D):=\inf\{\lambda\in\mathbb R\,|\,X(\overline K)^{\overline D}_{\leq\lambda}\text{ is Zariski dense}\}.\] 
The function $\widehat{\mu}_{\mathrm{ess}}(.)$ takes value in $\mathbb R\cup\{-\infty\}$.
Note that if $\overline D$ verifies the Dirichlet property, then the essential minimum of $h_{\overline D}(.)$ is non-negative.
\subsubsection{}\label{CV:measure}
We say that an adelic arithmetic $\mathbb R$-Cartier divisor $\overline D=(D,g)$ is \emph{relatively nef} if $D$ is a nef $\mathbb R$-Cartier divisor and $g_{v} $ is of $(C^0\cap\mathrm{PSH})$-type (for details, see \cite[\S~2.1]{MoAdel}). To each family $(\overline D_i)_{i=1}^{d+1}$ of relatively nef adelic arithmetic $\mathbb R$-Cartier divisor one can associate a real number $\widehat{\deg}(\overline D_1\cdots\overline D_{d+1})$ as in \cite[\S4.5]{MoAdel}\footnote{The smoothness condition for the scheme $X$ in \emph{loc. cit.} is actually not necessary.}. The intersection number function $(\overline D_1,\ldots,\overline D_{d+1})\mapsto\widehat{\deg}(\overline D_1,\ldots,\overline D_{d+1})$ is symmetric, additive and $\mathbb R_+$-homogeneous in each coordinate and hence extends to a $(d+1)$-linear form on the vector space of integrable\footnote{Recall that an adelic arithmetic $\mathbb R$-Cartier divisor is said to be \emph{integrable} if it can be written as the difference of two relatively nef adelic arithmetic $\mathbb R$-Cartier divisors.} adelic arithmetic $\mathbb R$-Cartier divisors. The extended function is continuous in each of its coordinates with respect to the topology on the vector space of all integrable adelic arithmetic $\mathbb R$-Cartier divisors defined by the usual convergence in each of its finite dimensional vector subspaces and the uniform convergence of Green functions. Therefore, for fixed integrable adelic arithmetic $\mathbb R$-Cartier divisors $\overline D_1,\ldots,\overline D_{d}$, the function $\overline D_{d+1}\mapsto\widehat{\deg}(\overline D_1,\ldots,\overline D_{d+1})$ can be extended by continuity to the whole vector space of adelic arithmetic $\mathbb R$-Cartier divisors on $X$.  

Let $\overline D_1,\ldots,\overline D_d$ be relatively nef adelic arithmetic $\mathbb R$-Cartier divisors. The intersection product defines a (non-negative) Radon measure $(\overline D_1\cdots\overline D_d)_v$ on $X^{\mathrm{an}}_{v}$ for each place $v\in M_K\cup K(\mathbb C)$ such that, for any $\phi\in C^0(X_{v}^{\mathrm{an}})$ one has
\[(\overline D_1\cdots\overline D_d)_v(\phi):=\widehat{\deg}(\overline D_1\cdots\overline D_d\cdot\overline O(\phi)).\]
More generally one can define a signed Borel measure $(\overline D_1\cdots\overline D_d)_v$ on $X_v^{\mathrm{an}}$ for integrable adelic arithmetic $\mathbb R$-Cartier divisors $\overline D_1,\ldots,\overline D_d$. Moreover, this signed measure is multi-linear in $\overline D_1,\ldots,\overline D_d$.  Note that in the case where $\overline D_1,\ldots,\overline D_d$ come from  adelic line bundles and $v$ is a non-archimedean place, the above measure has been constructed in \cite{CL06} (the archimedean case is more classical and relies on the theory of Monge-Amp\`ere operators). See \S\ref{Sec:intersectionnumber} \emph{infra.} for the integrability of Green functions with respect to this measure extending some results of \cite{Maillot00,CL_Thuillier}.

Let $r$ be an integer in $\{0,\ldots,d\}$ and  $\overline D_1,\ldots,\overline D_{r+1}$ be a family of integrable adelic arithmetic $\mathbb R$-Cartier divisors. If $Z$ is an $\mathbb R$-coefficient algebraic cycle of dimension $r$ in $X$, written into the linear combination of prime cycles as 
\[Z=\lambda_1Z_1+\cdots+\lambda_nZ_n.\]
Then we define the height of $Z$ with respect to $\overline D_1,\ldots,\overline D_{r+1}$ as 
\[h(\overline D_1,\ldots,\overline D_{r+1};Z):=\sum_{i=1}^n\lambda_i\widehat{\deg}(\overline D_1|_{Z_i}\cdots\overline D_{r+1}|_{Z_i}).\]
In the particular case where all $\overline D_i$ are equal to the same adelic arithmetic $\mathbb R$-Cartier divisor $\overline D$, we write $h(\overline D_0,\ldots,\overline D_r;Z)$ in abbreviation as $h(\overline D;Z)$. 
Note that when $Z$ is the algebraic cycle corresponding to a closed point $x$ of $X$, the height $h(\overline D;Z)$ equals $[K(x):K]h_{\overline D}(x)$. 

We say that an adelic arithmetic $\mathbb R$-Cartier divisor $\overline D=(D,g)$ is \emph{nef} if it is relatively nef and if the function $h_{\overline D}(\cdot)$ is non-negative. In the case where $\overline D$ is nef, the function $h(\overline D;\cdot)$ is non-negative on effective cycles.

\subsubsection{}\label{CT:effectivesection}
Let $\overline D$ be an adelic arithmetic $\mathbb R$-Cartier divisor on $X$. if $s \in \Rat(X)^{\times}_{\RR}$ and $D + (s) \geq 0$, then $|s|_{g_{v}} := |s|_{v} \exp(-g_{v}/2)$ is a continuous function on $X^{\an}_{v}$, where $v\in M_K\cup K(\mathbb C)$;
$|s|_{g_{v}} \leq 1$ for all $v$ if and only if $\overline{D} + \widehat{(s)} \geq 0$.

We denote by $H^0(X,D)$ the $K$-vector space 
\[\{\phi\in\Rat(X)^{\times}\,|\,D+(\phi)\geq 0\}\cup\{0\}\]
Assume that $s$ is an element in $H^0(X,D)$. For each $v\in M_K\cup K(\mathbb C)$, the Green function $g_{v}$ defines a continuous function $|s|_{g_{v}}$ such that \[|s|_{g_{v}}=|s|_{v}\exp(-g_{v}/2).\] This function vanishes on the locus of $\mathrm{div}(s)+D$. We also define
\[\|s\|_{v,\sup}:=\sup_{x\in X_{v}^{\mathrm{an}}}|s|_{v}(x).\]
Denote by $\hat{H}^0(X,\overline D)$ the set of all $s\in H^0(X,D)$ such that $\|s\|_{v,\sup}\leq 1$ for any $v\in M_K\cup K(\mathbb C)$. The \emph{arithmetic volume} of $\overline D$ is defined as 
\[\widehat{\vol}(\overline D):=\limsup_{n\rightarrow+\infty}\frac{\log\#\widehat{H}^0(X,n\overline D)}{n^{d+1}/(d+1)!}.\]

\section{Density of non-positive points}

This section is devoted to a non-denseness result for non-positive points under the Dirichlet property. This result will be useful in the following sections to construct counter-examples to the Dirichlet property. We fix a projective and geometrically integral scheme $X$ defined over a number field $K$. 

Let $S$ be a subset of $X(\overline{K})$.
For a proper subscheme $Y$ of $X$ and $v \in M_K \cup K(\CC)$,
we set
\[
\Delta(S;Y)^{\an}_{v} := \bigcup\nolimits_{x \in S \setminus Y(\overline{K})} 
\{ w^{\an} \mid w \in O_{v}(x) \}.
\]
The essential support
$\Supp_{\mathrm{ess}}(S)_{v}^{\an}$
of $S$ at $v$ is defined to be
\begin{equation}\label{Equ:esssupport}
\Supp_{\mathrm{ess}}(S)_{v}^{\an} := \bigcap_{Y \subsetneq X} 
\overline{\Delta(S; Y)^{\an}_{v}},
\end{equation}
where $\overline{\Delta(S;Y)^{\an}_{v}}$ is the closure of $\Delta(S;Y)^{\an}_{v}$ with respect to
the analytic topology.
Note that if $\Delta(S;\emptyset)^{\an}_{v}$ is dense with respect to
the analytic topology, then 
\[
\Supp_{\mathrm{ess}}(S)_{v}^{\an}
= X^{\an}_{v}.
\]
Moreover, if $S$ is not Zariski dense, then $\Supp_{\mathrm{ess}}(S)_{v}^{\an} = \emptyset$.

\begin{Lemma}[Non-denseness of non-positive points]
\label{lem:non:dense}
Let $\overline{D} = (D, g)$ be an adelic arithmetic $\RR$-Cartier divisor of $C^0$-type on $X$.
If $s$ is an element of $\Rat(X)^{\times}_{\RR}$ with
$\overline{D} + \widehat{(s)} \geq 0$, then
\[
\Supp_{\mathrm{ess}}\big(X(\overline{K})^{\overline{D}}_{\leq 0}\big)_{\!v}^{\!\an} \cap
\{ x \in X^{\an}_{v} \mid \vert s \vert_{g_{v}}(x) < 1 \} = \emptyset
\]
for all $v \in M_K \cup K(\CC)$.
In particular, if $\Supp(D + (s)) \not= \emptyset$,
then $\Delta\big(X(\overline{K})^{\overline{D}}_{\leq 0}; \emptyset\big)^{\!\an}_{\!v}$ is not dense with respect to
the analytic topology.
\end{Lemma}

\begin{proof}
We set $S := X(\overline{K})^{\overline{D}}_{\leq 0}$, $Y := \Supp(D + (s))$ and
$g'_{v} := -\log \vert s \vert_{g_{v}}^2$.
By our assumption,
$g'_{v} \geq 0$ for all $v \in M_K \cup K(\CC)$.

\begin{Claim}
For all $y \in \Delta(S;Y)^{\an}_{v}$, 
we have $g'_{v}(y) = 0$.
\end{Claim}

\begin{proof}
For $y \in \Delta(S;Y)^{\an}_{v}$,
we choose $x \in S \setminus Y(\overline{K})$
such that $y = w^{an}$ for some
$w \in O_{v}(x)$.
Then,
\[
0 \geq 2 [K(x) : K] h_{\overline{D}}(x) =
2 \adeg(\srest{\overline{D} + \widehat{(s)}}{x}) =
\sum_{v' \in M_K \cup K(\CC)} \sum_{w' \in O_{v'}(x)} g'_{v'}({w'}^{\an}).
\]
As $g'_{v'} \geq 0$ for all $v' \in M_K \cup K(\CC)$, the assertion follows.
\end{proof}

We assume that 
$\Supp_{\mathrm{ess}}(S)_{v}^{\an} \cap
\{ x \in X^{\an}_{v} \mid \vert s \vert_{g_{v}}(x) < 1 \} \not= \emptyset$.
In particular,
\[
\overline{\Delta(S;Y)^{\an}_{v}} \cap
\{ x \in X^{\an}_{v} \mid \vert s \vert_{g_{v}}(x) < 1 \} \not= \emptyset.
\]
We can choose $y_{\infty} \in X^{\an}_{v}$ and
a sequence $\{ y_m \}$ in $X^{\an}_{v}$ such that
$\vert s \vert_{g_{v}}(y_{\infty}) < 1$, 
$y_m \in \Delta(S;Y)^{\an}_{v}$
for all $m$ and
$\lim_{m\to\infty} y_m = y_{\infty}$. 
By the above claim, $\vert s \vert_{g_{v}}(y_m) = 1$ for all $m$, and hence
$\vert s \vert_{g_{v}}(y_{\infty}) = \lim_{m\to\infty} \vert s \vert_{g_{v}}(y_m) = 1$.
This is a contradiction.
\end{proof}

\begin{Remark}
\label{rem:ess:support:dynamics}
Let $f : \PP^1_K \to \PP^1_K$ be a surjective endomorphism over $K$ with
$\deg(f) \geq 2$. Let $S$ be the set of periodic points in $\PP^1(\overline{K})$.
Fix $\sigma \in K(\CC)$. By Lemma~\ref{lem:preperiodic:dim:0},
$\Delta(S;\emptyset)_{\sigma}^{\an}$ coincides with
the set of periodic points of $f_{\sigma}$ in $(\PP^1_K)_{\sigma}^{\an}$.
The following are  well-known:
\begin{enumerate}
\renewcommand{\labelenumi}{(\arabic{enumi})}
\item
The closure of the set of repelling periodic points is the Julia set $J_{\sigma}$ of $f_{\sigma}$
(\cite[Theorem~4.2.10]{HolDyn}).

\item
The set of non-repelling periodic points is a finite set (\cite[Theorem~4.2.9]{HolDyn}).

\item
The Julia set $J_{\sigma}$ is closed and perfect, that is, $J_{\sigma}$ is closed and $J_{\sigma}$ has no isolated points in $J_{\sigma}$ (\cite[Theorem~2.3.6]{HolDyn}).
\end{enumerate}
Therefore, we can see that the essential support $\Supp_{\mathrm{ess}}(S)^{\an}_{\sigma}$
of $S$ at $\sigma$ is equal to the Julia set $J_{\sigma}$.
\end{Remark}

For a subset $S$ of $X(\overline{K})$ and $v \in M_K \cup K(\CC)$,
we set $S_v = \bigcup_{x \in S} O_v(x)$.
Let us consider a way to give the essential support of $S$ at $v$ in terms of $S_v$ and $X_v$.

\begin{Proposition}
\label{prop:essential:supp:v}
$\Supp_{\mathrm{ess}}(S)^{\an}_{v} = \bigcap_{Y_v \subsetneq X_v} \overline{\{ w^{\an} \mid w \in S_v \setminus Y_v(\overline{K_v}) \}}$,
where $Y_v$ runs over all proper subschemes of $X_v$.
\end{Proposition}

\begin{proof}
It is sufficient to show that, for a proper subschemes $Y_v$ of $X_v$, there is a
proper subscheme $Y$ of $X$ such that
\[
\bigcup\nolimits_{x \in S \setminus Y(\overline{K})} O_v(x)  \subseteq S_v \setminus Y_v(\overline{K_v}),
\]
that is, $S_v \cap Y_v(\overline{K_v}) \subseteq \bigcup_{x \in S \cap Y(\overline{K})} O_v(x)$.

Let $\pi : X_v \to X$ be the projection.
For $x \in X(\overline{K})$ and $w \in X_v(\overline{K_v})$,
the natural induced morphisms 
\[
\Spec(\overline{K}) \to X_{\overline{K}}
\quad\text{and}\quad
\Spec(\overline{K_v}) \to (X_{v})_{\overline{K_v}}
\]
are denoted by
$\tilde{x}$ and $\tilde{w}$, respectively,
where 
\[
X_{\overline{K}} = X \times_{\Spec(K)} \Spec(\overline{K})
\quad\text{and}\quad
(X_{v})_{\overline{K_v}} = X_v \times_{\Spec(K_v)} \Spec(\overline{K_v}).
\]
We fix an embedding $\overline{K} \hookrightarrow \overline{K_v}$, which yields 
a morphism $\bar{\pi} : (X_v)_{\overline{K_v}} \to X_{\overline{K}}$.
In the case where $w \in O_v(x)$, there is a homomorphism $\iota_w : \overline{K} \to \overline{K_v}$ over $K$
such that the following diagram is commutative:
\[
\begin{CD}
\Spec(\overline{K_v}) @>{\tilde{w}}>> (X_{v})_{\overline{K_v}} \\
@V{\iota_w^a}VV @VV{\bar{\pi}}V \\
\Spec(\overline{K}) @>{\tilde{x}}>> X_{\overline{K}}
\end{CD}
\]
Let $\widetilde{K}$ be the algebraic closure of $K$ in $\overline{K_v}$.
Note that $\iota_w(\overline{K}) = \widetilde{K}$.

Let $D$ be a Cartier divisor on $X$ such that $X^{\circ} := X \setminus \Supp(D)$ is affine.
Let $A$ be a finitely generated $K$-algebra with $X^{\circ} = \Spec(A)$.
Note that 
\[
S_v \cap \Supp(D)_v(\overline{K_v}) = \bigcup\nolimits_{x \in S \cap \Supp(D)(\overline{K})} O_v(x).
\]
If $Y_v \subseteq \Supp(D)_v$, then the assertion is obvious, so that we may assume that
$Y_v \not\subseteq \Supp(D)_v$. 
We put $T = S_v \cap (Y_v(\overline{K_v}) \setminus \Supp(D)_v(\overline{K_v}))$.

\begin{Claim}
There is a non-zero $h \in A \otimes_{K} \overline{K}$ such that $\tilde{w}^*(\bar{\pi}^*(h)) = 0$ for all $w \in T$.
\end{Claim}

\begin{proof}
Let $I_v$ be the ideal of $A \otimes_K K_v$ defining
$Y_v \cap X_v^{\circ}$.
Choose a non-zero element $h'$ of $I_v$. There are $h_1, \ldots, h_r \in A \otimes_K \overline{K}$ and
$a_1, \ldots, a_r \in \overline{K_v}$ such that
\[
h' = a_1 \bar{\pi}^*(h_1) + \cdots + a_r \bar{\pi}^*(h_r)
\]
and $a_1, \ldots, a_r$ are linearly independent over $\widetilde{K}$.
For $w \in T$ and $w \in O_v(x)$, by using the above diagram,
\begin{multline*}
0 = \tilde{w}^*(h') = a_1 \tilde{w}^*(\bar{\pi}^*(h_1)) + \cdots + a_r \tilde{w}^*(\bar{\pi}^*(h_r)) \\
=
a_1 \iota_w(\tilde{x}^*(h_1)) + \cdots + a_r \iota_w(\tilde{x}^*(h_r)),
\end{multline*}
so that $\tilde{x}^*(h_1) = \cdots = \tilde{x}^*(h_r) = 0$.
Therefore, the assertion follows.
\end{proof}

We set $h = c_1 h_1 + \cdots + c_l h_l$ for some $c_1, \ldots, c_l \in \overline{K}$ and $h_1, \ldots, h_l \in A$.
Let $K'$ be a finite Galois extension of $K$ such that $K(c_1, \ldots, c_l) \subseteq K'$.
Here we put $f = \prod_{\sigma \in \Gal(K'/K)} \sigma(h)$.
Note that $f \in A \setminus \{ 0 \}$ and $w^*(\pi^*(f)) = 0$ for all $w \in T$, so that
$T \subseteq \bigcup_{S \cap \Spec(A/fA)(\overline{K})} O_v(x)$. Therefore,
if we set $Y = \Supp(D) \cup \Spec(A/fA)$, then the proposition follows.
\end{proof}


\renewcommand{\theTheorem}{\arabic{section}.\arabic{subsection}.\arabic{Theorem}}
\renewcommand{\theClaim}{\arabic{section}.\arabic{subsection}.\arabic{Theorem}.\arabic{Claim}}

\section{Endomorphism and Green function}
\label{sec:end:green}

This section consists of the construction of the canonical Green functions for a given $\mathbb R$-Cartier divisor in the algebraic dynamical system setting, which can be considered as a generalization of the construction of the canonical metrics in \cite{ZhSmall}. Here we explain them in terms of Green functions on either Berkovich spaces or complex varieties.
Throughout this section, we fix the following notation.
Let $X$ be a projective and geometrically integral variety over a field $K$.
Let $f : X \to X$ be a surjective endomorphism of $X$ over $K$.
Let $D$ be an $\RR$-Cartier divisor on $X$.
We assume that there are a real number $d$ and $\varphi \in \Rat(X)^{\times}_{\RR}$
such that $d > 1$ and $f^*(D) = dD + (\varphi)$.

\subsection{
Non-archimedean case}
\setcounter{Theorem}{0}
We assume that $K$ is the quotient field of a complete discrete valuation ring $R$.
Let $X^{\an}$ be the analytification of $X$ in the sense of Berkovich.
Note that $f^{\an} : X^{\an} \to X^{\an}$ is also surjective by \cite[Proposition~3.4.7]{Be}.
Let us begin with the following proposition:

\begin{Proposition}
\label{prop:exist:Green:fun}
There exists a unique $D$-Green function $g$ of $C^0$-type on $X^{\an}$
such that $(f^{\an})^*(g) = dg - \log \vert \varphi \vert^2$.
\end{Proposition}

\begin{proof}
Let us fix a $D$-Green function $g_0$ of $C^0$-type on $X^{\an}$.
As 
\[
f^*(D) = d D + (\varphi), 
\]
$(f^{\an})^*(g_0)$ is a $(dD + (\varphi))$-Green function of $C^0$-type, and hence,
there is a continuous function $\lambda_0$ on $X^{\an}$ such that
\[
(f^{\an})^*(g_0) = d g_0 - \log \vert \varphi \vert^2 + \lambda_0.
\]
For each $n \in \ZZ_{\geq 1}$,
let us consider a continuous function $h_n$ on $X^{\an}$ given by
\[
h_n := \sum_{i=0}^n \frac{1}{d^{i+1}} ((f^{\an})^i)^*(\lambda_0).
\]

\begin{Claim}
\begin{enumerate}
\renewcommand{\labelenumi}{(\alph{enumi})}
\item
There is a continuous function $h$ on $X^{\an}$ such that
the sequence $\{ h_n \}$ converges uniformly to $h$.

\item
$(f^{\an})^*(h) + \lambda_0 = dh$.
\end{enumerate}
\end{Claim}

\begin{proof}
(a) It is sufficient to show that $\Vert h_n - h_m \Vert_{\sup} \to 0$ as
$n, m \to \infty$. 
Indeed, if $n > m$, then
\begin{align*}
\Vert h_n - h_m \Vert_{\sup} & = \left\Vert \sum_{i=m+1}^n \frac{1}{d^{i+1}} ((f^{\an})^i)^*(\lambda_0) \right\Vert_{\sup}
\leq \sum_{i=m+1}^n \frac{1}{d^{i+1}} \left\Vert ((f^{\an})^i)^*(\lambda_0) \right\Vert_{\sup} \\
& = \left\Vert \lambda_0 \right\Vert_{\sup} \sum_{i=m+1}^n \frac{1}{d^{i+1}},
\end{align*}
as required.

(b) Note that
\[
(f^{\an})^*(h_n) + \lambda_0 = \sum_{i=0}^n \frac{1}{d^{i+1}} ((f^{\an})^{i+1})^*(\lambda_0) + \lambda_0
= d h_{n+1},
\]
and hence the assertion follows.
\end{proof}

If we set $g = g_0 + h$, then $g$ is a $D$-Green function of $C^0$-type and
\begin{align*}
(f^{\an})^*(g) & = (f^{\an})^*(g_0) + (f^{\an})^*(h)
= \left( d g_0 - \log \vert \varphi \vert^2 + \lambda_0 \right) + \left(dh - \lambda_0\right)  \\
& = dg - \log \vert \varphi \vert^2,
\end{align*}
as desired.

\medskip
Next we consider the uniqueness of $g$.
Let $g'$ be another $D$-Green function of $C^0$-type such that
$(f^{\an})^*(g') = dg' - \log \vert \varphi \vert^2$.
Then, as $g'-g$ is a continuous function on $X^{\an}$ and
$(f^{\an})^*(g'-g) = d(g'-g)$, we have
\[
\Vert g' - g \Vert_{\sup} = \Vert (f^{\an})^*(g'-g) \Vert_{\sup} = \Vert d(g'-g) \Vert_{\sup}
= d \Vert g'-g \Vert_{\sup},
\]
and hence $\Vert g'-g \Vert_{\sup} = 0$. Thus $g' = g$.
\end{proof}

\begin{Proposition}
\label{prop:Green:model}
Let $\XXX \to \Spec(R)$ be a model of $X$ over $\Spec(R)$ and $\DDD$ an $\RR$-Cartier divisor on $\XXX$ such that
$\DDD$ coincides with $D$ on $X$. 
If there is an endomorphism $\tilde{f} : \XXX \to \XXX$ over $\Spec(R)$ such that
$\rest{\tilde{f}}{X} = f$ and
$\tilde{f}^*(\DDD) = d\DDD + (\varphi)_{\XXX}$, then the $D$-Green function
$g_{(\XXX, \DDD)}$ arising from the model $(\XXX, \DDD)$ 
is equal to $g$.
\end{Proposition}

\begin{proof}
The relation $\tilde{f}^*(\DDD) = d\DDD + (\varphi)_{\XXX}$ yields
\[
(f^{\an})^*(g_{(\XXX,\DDD)}) = dg_{(\XXX,\DDD)} - \log \vert \varphi \vert^2,
\]
so that, by the uniqueness of $g$, we have $g_{(\XXX,\DDD)} = g$.
\end{proof}

Using the identities
\[
f^*(D) = dD +  \left( \varphi \right)
\quad\text{and}\quad
(f^{\an})^*(g) = g - \log \left| \varphi \right|^2,
\]
we can easily see that
\[
(f^n)^*(D) = d^n D + \left( \varphi_n \right)
\quad\text{and}\quad
((f^{\an})^n)^*(g) = d^n g - \log \left| \varphi_n \right|^2
\]
for $n \geq 1$, where
\[
\varphi_n := \prod_{i=0}^{n-1}\left( (f^{n-1-i})^*(\varphi) \right)^{d^i}.
\]
Let $\XXX \to \Spec(R)$ be a model of $X$ over $\Spec(R)$ and $\DDD$ an $\RR$-Cartier divisor on $\XXX$ with
$\rest{\DDD}{X} = D$.
For each $n \geq 1$, we choose a model $\XXX_n \to \Spec(R)$ of $X$ over $\Spec(R)$ together with
a morphism $\tilde{f}_n : \XXX_n \to \XXX$ over $\Spec(R)$ such that
$\rest{\tilde{f}_n}{X} = f^n$.
Here we define an $\RR$-Cartier divisor $\DDD_n$ on $\XXX_n$ to be
\[
\DDD_n := \frac{1}{d^n} \left( \tilde{f}_n^*(\DDD) - \left( \varphi_n \right)_{\XXX_n}
\right).
\]
Note that $\rest{\DDD_n}{X} = D$.
Then we have the following:

\begin{Proposition}
\label{prop:approx:Green}
If we set $\theta_n = g - g_{(\XXX_n, \DDD_n)}$, then $\lim_{n\to\infty} \Vert \theta_n \Vert_{\sup} = 0$.
In particular, if $\DDD$ is relatively nef, then $g$ is of $(C^0 \cap \Tpsh)$-type.
\end{Proposition}

\begin{proof}
Since
\[
\tilde{f}_n^*(\DDD) = d^n \DDD_n + \left( \varphi_n \right)_{\XXX_n},
\]
we have 
\[
((f^{\an})^n)^*(g_{(\XXX, \DDD)}) = d^n g_{(\XXX_n,\DDD_n)} -
\log \left| \varphi_n\right|^2,
\]
so that if we set $\theta = g - g_{(\XXX, \DDD)}$, then
$((f^{\an})^n)^*(\theta) = d^n \theta_n$.
Therefore,
\[
\Vert \theta \Vert_{\sup} = \Vert ((f^{\an})^n)^*(\theta) \Vert_{\sup} = 
\Vert d^n \theta_n \Vert_{\sup} = d^n \Vert \theta_n \Vert_{\sup},
\]
and hence $\lim_{n\to\infty} \Vert \theta_n \Vert_{\sup} = 0$.

For the last statement, note that if $\DDD$ is relatively nef,
then $\DDD_n$ is also relatively nef for $n \geq 1$.
\end{proof}

\subsection{Complex case}
\setcounter{Theorem}{0}
We assume that $K = \CC$.

\begin{Proposition}
\label{prop:exist:Green:fun:complex}
There exists a unique $D$-Green function $g$ of $C^0$-type on $X$
such that $f^*(g) = dg - \log \vert \varphi \vert^2$.
Moreover,
if there is a $D$-Green function of $(C^0 \cap \Tpsh)$-type,
then $g$ is also of $(C^0 \cap \Tpsh)$-type.
\end{Proposition}

\begin{proof}
We can prove the unique existence of $g$ in the same way as Proposition~\ref{prop:exist:Green:fun}.
Let $g_0$ be a $D$-Green function of $(C^0 \cap \Tpsh)$-type.
As in the previous subsection, we can see
\[
(f^n)^*(D) = d^n D + \left( \varphi_n \right)
\quad\text{and}\quad
(f^n)^*(g) = d^n g - \log \left| \varphi_n \right|^2
\]
for $n \geq 1$, where
\[
\varphi_n := \prod_{i=0}^{n-1}\left( (f^{n-1-i})^*(\varphi) \right)^{d^i}.
\]
Here we define $g_n$ to be
\[
g_n := \frac{1}{d^n} \left( (f^n)^*(g_0) + \log \left|\varphi_n \right|^2 \right),
\]
that is,
\[
(f^n)^*(g_0) = d^n g_n- \log \left| \varphi_n \right|^2.
\]
Then $g_n$ is a $D$-Green function of $(C^0 \cap \Tpsh)$-type. Moreover,
if we set $\theta = g - g_0$ and $\theta_n = g - g_n$, then
$(f^n)^*(\theta) = d^n \theta_n$.
Thus
\[
\Vert \theta \Vert_{\sup} = \Vert (f^n)^*(\theta) \Vert_{\sup} = 
\Vert d^n \theta_n \Vert_{\sup} = d^n \Vert \theta_n \Vert_{\sup},
\]
and hence $\lim_{n\to\infty} \Vert \theta_n \Vert_{\sup} = 0$. Therefore,
$g$ is of $(C^0 \cap \Tpsh)$-type by \cite[Theorem~2.9.14, (iii)]{MK}.
\end{proof}

Let $c : \Spec(\CC) \to \Spec(\CC)$ be the morphism given by 
the complex conjugation map 
$z \mapsto \bar{z}$.
Let $\widetilde{X}$ denote the fiber product $X \times^{c}_{\Spec(\CC)} \Spec(\CC)$ in terms of $c$.
Let $F : \widetilde{X} \to X$ be the projection morphism and
$\tilde{f} : \widetilde{X} \to \widetilde{X}$ the induced morphism by $f$.
Note that the following diagram is commutative:
\[
\xymatrix{
\widetilde{X} \ar[r]^{\tilde{f}} \ar[d]_{F} &  \widetilde{X} \ar[d]^{F} \\
X \ar[r]^{f} & X
}
\]
If we set $\widetilde{D} = F^*(D)$ and $\tilde{\varphi} = F^*(\varphi)$,
then $\tilde{f}^*(\widetilde{D}) = d \widetilde{D} + (\tilde{\varphi})$.
For $x \in \widetilde{X}(\CC)$,
the composition $\Spec(\CC) \overset{c}{\longrightarrow} \Spec(\CC) 
\overset{x}{\longrightarrow} \widetilde{X} \overset{F}{\longrightarrow} X$
yields a $\CC$-valued point of $X$, so that we define
$F_{\infty} : \widetilde{X}(\CC) \to X(\CC)$ to be $F_{\infty}(x) = F \circ x \circ c$.
The above commutative diagram gives rise to the following commutative diagram:
\[
\xymatrix{
\widetilde{X}(\CC) \ar[r]^{\tilde{f}} \ar[d]_{F_{\infty}} &  \widetilde{X}(\CC) \ar[d]^{F_{\infty}} \\
X(\CC) \ar[r]^{f} & X(\CC)
}
\]

\begin{Proposition}
\label{prop:Green:fun:complex:F:infty}
Let $g$ be a $D$-Green function of $C^0$-type on $X$ with $f^*(g) = dg - \log | \varphi |^2$
as in Proposition~\rom{\ref{prop:exist:Green:fun:complex}}.
Then $\tilde{g} := F_{\infty}^*(g)$ is a $\widetilde{D}$-Green function of $C^0$-type on $\widetilde{X}$
with $\tilde{f}^*(\tilde{g}) = d\tilde{g} - \log | \tilde{\varphi} |^2$.
Moreover, if $g$ is of $(C^0 \cap \Tpsh)$-type, then
$\tilde{g}$ is also of $(C^0 \cap \Tpsh)$-type.
\end{Proposition}

\begin{proof}
It is easy to see that
$\tilde{g}$ is a $\widetilde{D}$-Green function of $C^0$-type on $\widetilde{X}$ because
\[
F^*(\psi)(x) = x^*(F^*(\psi)) = c^*(c^*(x^*(F^*(\psi)))) = \overline{(F_{\infty}(x))^*(\psi)}
= \overline{\psi(F_{\infty}(x))}
\]
for $x \in \widetilde{X}(\CC)$ and $\psi \in \Rat(X)^{\times}$.
In addition,
\[
\tilde{f}^*(\tilde{g}) = F_{\infty}^*(f^*(g)) = F_{\infty}^*(dg - \log | \varphi |^2) = d\tilde{g} - \log | \tilde{\varphi} |^2.
\]
The last assertion follows from the same argument of \cite[Lemma~5.1.1]{MoArZariski}.
\end{proof}

\renewcommand{\theTheorem}{\arabic{section}.\arabic{Theorem}}
\renewcommand{\theClaim}{\arabic{section}.\arabic{Theorem}.\arabic{Claim}}

\section{Canonical compactification}
\label{sec:canonical:comp}
Let $X$ be a projective and geometrically integral variety over a number field $K$.
Let $f : X \to X$ be a surjective endomorphism of $X$ over $K$.
Let $D$ be an $\RR$-Cartier divisor on $X$.
We assume that there are a real number $d$ and $\varphi \in \Rat(X)^{\times}_{\RR}$
such that $d > 1$ and $f^*(D) = dD + (\varphi)$.
We use the same notation as in Conventions and terminology~\ref{CV:residue:field} $\sim$ \ref{CV:analytic:point}.
In addition, for each $v \in M_K \cup K(\CC)$,
let $f^{\an}_{v} : X^{\an}_{v} \to X^{\an}_{v}$ be the induced map by $f$.
By Proposition~\ref{prop:exist:Green:fun}, for $\mathfrak p \in M_K$,
there is a unique $D$-Green function $g_{\mathfrak p}$ of $C^0$-type on $X_{\mathfrak p}^{\an}$ with
\[
(f^{\an}_{\mathfrak p})^*(g_{\mathfrak p}) = d g_{\mathfrak p} - \log \vert \varphi \vert_{\mathfrak p}^2.
\]
We can find a model $\XXX_U$ of $X$ over a non-empty Zariski open set $U$ of $\Spec(O_K)$,
an $\RR$-Cartier divisor $\DDD_U$ on $\XXX_U$ and
an endomorphism $f_U : \XXX_U \to \XXX_U$ over $U$
such that $\rest{f_U}{X} = f$ and $f_U^*(\DDD) = d\DDD + (\varphi)$ on $\XXX_U$, so that,
by Proposition~\ref{prop:Green:model}, for $P \in U$,
$g_P$ comes from the model $(\XXX_U, \DDD_U)$.
Further,
by virtue of Proposition~\ref{prop:exist:Green:fun:complex} and Proposition~\ref{prop:Green:fun:complex:F:infty},
let us take a unique $F_{\infty}$-invariant $D$-Green function $g_{\infty}$ of $C^0$-type on $X^{\an}_{\infty}$ 
(for the definition of $X^{\an}_{\infty}$,
see Conventions and terminology~\ref{CV:analytication}) such that
\[
(f^{\an}_{\infty})^*(g_{\infty}) = d g_{\infty} - \log \vert \varphi \vert_{\infty}^2,
\]
where $f^{\an}_{\infty} := \coprod_{\sigma \in K(\CC)} f^{\an}_{\sigma}$.
Therefore,
\[
\overline{D} := \left(D, \sum_{P \in M_K} g_P [P] + g_{\infty}[\infty]\right)
\]
forms an adelic arithmetic Cartier divisor of $C^0$-type on $X$.
By our construction,
\[
f^*(\overline{D}) = d \overline{D} + \widehat{(\varphi)}.
\]
The adelic arithmetic Cartier divisor $\overline{D}$ is called
{\em the canonical compactification of $D$} with respect to $f$.

\begin{Lemma}
\label{lem:can:comp:nef}
If $D$ is ample, that is,
there are ample Cartier divisors $D_1, \ldots, D_r$ on $X$ and $a_1, \ldots, a_r \in \RR_{>0}$ with
$D = a_1 D_1 + \cdots + a_r D_r$, then $\overline{D}$ is nef.
\end{Lemma}

\begin{proof}
First let us see the following claim:

\begin{Claim}
\begin{enumerate}
\renewcommand{\labelenumi}{(\alph{enumi})}
\item
There are
a model $\pi : \XXX \to \Spec(O_K)$ of $X$ over $\Spec(O_K)$
and a relatively nef $\RR$-Cartier divisor $\DDD$ on $\XXX$ such that $\rest{\DDD}{X} = D$. 

\item
There is an $F_{\infty}$-invariant 
$D$-Green function $h$ of $C^{\infty}$-type on $X^{\an}_{\infty}$ such that $c_1(D, h)$ is positive.
\end{enumerate}
\end{Claim}

\begin{proof}
If $D$ is an ample Cartier divisor, then the assertions (a) and (b) are well-known.
Moreover, in this case, $\DDD$ in (a) can be taken as a $\QQ$-Cartier divisor.

(a) 
For each $i=1, \ldots, r$,
there are a model $\XXX_i \to \Spec(O_K)$ of $X$ over $\Spec(O_K)$
and a relatively nef $\QQ$-Cartier divisor $\DDD_i$ on $\XXX_i$ such that $\rest{\DDD_i}{X} = D_i$. 
Let us take a model $\XXX \to \Spec(O_K)$ of $X$ over $\Spec(O_K)$ such that
we have a birational morphism $\mu_i : \XXX \to \XXX_i$ over $\Spec(O_K)$ for each $i=1, \ldots, r$.
If we set $\DDD = a_1 \mu_1^*(\DDD_1) + \cdots + a_r \mu_r^*(\DDD_r)$,
then $\DDD$ is relatively nef and $\rest{\DDD}{X} = D$.

(b) For each $i=1, \ldots, r$,
let $h_i$ be an $F_{\infty}$-invariant $D_i$-Green function of $C^{\infty}$-type on $X^{\an}_{\infty}$
such that $c_1(D_i, h_i)$ is positive.
Then $a_1 h_1 + \cdots + a_r h_r$ is our desired Green function.
\end{proof}

By the above claim together with Proposition~\ref{prop:approx:Green}
and Proposition~\ref{prop:exist:Green:fun:complex},
$g_P$ and $g_{\infty}$ are  of $(C^0 \cap \Tpsh)$-type on $X_P^{\an}$ and $X^{\an}_{\infty}$, respectively.
Therefore, $\overline{D}$ is
relatively nef.
Let $h_{\overline{D}}$
be the height function 
associated with $\overline{D}$.
Then 
\[
h_{\overline{D}}(f(x)) = d h_{\overline{D}}(x)
\]
for all $x \in X(\overline{K})$.
Indeed, 
\[
h_{\overline{D}}(f(x)) = h_{f^*(\overline{D})}(x) = h_{d \overline{D} + \widehat{(\varphi)}}(x)
= h_{d \overline{D}}(x) = d  h_{\overline{D}}(x).
\]
As $D$ is ample, there is a constant $C$ such that $h_{\overline{D}} \geq C$.
In particular, 
\[
h_{\overline{D}}(x) = h_{\overline{D}}(f^n(x))/d^n \geq C/d^n
\]
for all $n \geq 1$, and hence $h_{\overline{D}}(x) \geq 0$ for $x \in X(\overline{K})$.
Therefore, $\overline{D}$ is nef.
\end{proof}

For $v \in M_K \cup K(\CC)$, 
we set
\[
\begin{cases}
\Prep(f) := \{ x \in X(\overline{K}) \mid \text{$f^n (x) = f^m(x)$ for some $0 \leq n < m$} \}, \\
\Prep(f_{v}) := \{ x \in X_{v}(\overline{K_{v}}) \mid \text{$f_{v}^n (x) = f_{v}^m(x)$ 
for some $0 \leq n < m$} \}.
\end{cases}
\]
An element of $\Prep(f)$ (resp. $\Prep(f_{v})$) is called
a {\em preperiodic point} of $f$ (resp. $f_{v}$).
Moreover, for a subset $T$ of $X_{v}(\overline{K_{v}})$, $T^{\an}$ is defined by
\[
T^{\an} := \{ w^{\an} \mid w \in T \}
\]
(for the definition of $w^{\an}$, see Conventions and terminology~\ref{CV:analytic:point}).
Let us see the following proposition.

\begin{Proposition}
\label{prop:thm:canonical:comp}
$\bigcup_{x \in \Prep(f)} O_{v}(x) = \Prep(f_{v})$
(for the definition of
$O_{v}(x)$, see Conventions and terminology~\rom{\ref{CV:orbit}}).
\end{Proposition}

\begin{proof}
Clearly $\bigcup_{x \in \Prep(f)} O_{v}(x) \subseteq \Prep(f_{v})$.
Conversely, we suppose that $x \in \Prep(f_{v})$, that is,
$f_{v}^m \circ x = f_{v}^n \circ x$ for some $m > n \geq 0$.
Let $\pi_{v} : X_{v} \to X$ be the projection.
Then $\pi_{v} \circ f_{v}^m \circ x =\pi_{v} \circ f_{v}^n \circ x$.
Note that the following diagram is commutative:
\[
\xymatrix{
X_{v} \ar[r]^{f_{v}} \ar[d]_{\pi_{v}} &  X_{v} \ar[d]^{\pi_{v}} \\
X \ar[r]^{f}  & X,
}
\]
so that we have $ f^m \circ \pi_{v} \circ x = f^n \circ \pi_{v} \circ x$.
Therefore, Lemma~\ref{lem:preperiodic:dim:0} below, 
there are a closed point $\xi$ of $X$ and a homomorphism
$K(\xi) \to \overline{K_v}$ such that $\pi_{v} \circ x$ is given by 
the composition $\Spec(\overline{K_v}) \to \Spec(K(\xi)) \to X$, so that
the assertion follows.
\end{proof}

\begin{Lemma}
\label{lem:preperiodic:dim:0}
Let $V$ be a projective variety over a field $k$.
Let $f : V \to V$ and $g : V \to V$ be surjective endomorphisms
of $V$ and let $D$ be an $\RR$-Cartier divisor on $V$.
We assume the following:
\begin{enumerate}
\renewcommand{\labelenumi}{(\arabic{enumi})}
\item
$D$ is ample, that is, there are ample Cartier divisors $D_1, \ldots, D_r$ on $V$ and
$a_1, \ldots, a_r \in \RR_{>0}$ with $D = a_1 D_1 + \cdots + a_r D_r$.

\item
There are $\phi, \psi \in \Rat(V)^{\times}_{\RR}$ and
$a, b \in \RR_{>0}$ such that
$f^*(D) = a D + (\phi)$, $g^*(D) = b D + (\psi)$ and $a \not= b$.
\end{enumerate}
If $\Omega$ is a field over $k$, $x \in V(\Omega)$ and $f(x) = g(x)$,
then there are a closed point $\xi$ of $V$ and a homomorphism $k(\xi) \to \Omega$
such that $x$ coincides with the composition of $\Spec(\Omega) \to \Spec(k(\xi)) \to V$.
\end{Lemma}

\begin{proof}
Let $Z$ denote $(f \times g)^{-1}(\Delta)$,
where $f \times g : V \to V \times V$ is a morphism given by
$x \mapsto (f(x), g(x))$ and $\Delta$ is the diagonal in $V \times V$.
It is sufficient to show that $\dim Z \leq 0$.
We assume the contrary, so that we can find
a $1$-dimensional subvariety $C$ of $V$ with $C \subseteq Z$.
Then $\rest{f}{C} = \rest{g}{C}$. In particular, $f_*(C) = g_*(C)$.
As 
\[
f^*(D) - g^*(D) = (a- b)D + (\phi\psi^{-1}),
\]
we have
\[
(a-b) (D \cdot C) = ((f^*(D) - g^*(D)) \cdot C) = (D \cdot f_*(C)) - (D \cdot g_*(C)) = 0,
\]
and hence $(D \cdot C) = 0$.
This is a contradiction because $D$ is ample.
\end{proof}

The purpose of this section is to prove the following theorem:

\begin{Theorem}
\label{thm:canonical:comp}
We assume that $D$ is ample.
If there are $v \in M_K \cup K(\CC)$ and a subvariety
$Y_{v} \subseteq X_{v}$ such that
$\dim Y_{v} \geq 1$ and 
$Y_{v} \subseteq \Supp_{\mathrm{ess}}(\Prep(f))_v^{\an}$,
then
the Dirichlet property of $\overline{D}$ does not hold.
In particular, if $\Prep(f_{v})^{\an}$ is dense in $X^{\an}_{v}$ with respect to the analytic
topology for some $v \in M_K \cup K(\CC)$,
then the Dirichlet property of $\overline{D}$ does not hold.
\end{Theorem}

\begin{proof}

We assume that $\overline{D} + \widehat{(s)}$ is effective 
for some $s \in \Rat(X)^{\times}_{\RR}$.
Here we set $S := X(\overline{K})^{\overline{D}}_{\leq 0}$
(for the definition of $X(\overline{K})^{\overline{D}}_{\leq 0}$, see Section~\ref{sec:non:dense}).
By Lemma~\ref{lem:non:dense},
\[
\Supp_{\mathrm{ess}}\big(S\big)^{\an}_v \cap
\{ x \in X^{\an}_{v} \mid \vert s \vert_{g_{v}}(x) < 1 \} = \emptyset.
\]
Note that
if $x \in \Prep(f)$ for $x \in X(\overline{K})$, then $h_{\overline{D}}(x) = 0$.
Therefore, 
$\Prep(f) \subseteq S$, and hence
\[
\Supp_{\mathrm{ess}}\big(\Prep(f)\big)^{\an}_v \cap
\Supp(D + (s))_{v}^{\an} = \emptyset
\]
because $\Supp(D + (s))_{v}^{\an} \subseteq \{ x \in X^{\an}_{v} \mid \vert s \vert_{g_{v}}(x) < 1 \}$.
As $(D + (s))_{v}$ is ample, we can see that $Y_{v} \cap \Supp(D + (s))_{v} \not= \emptyset$.
In particular,
\[
Y_{v}^{\an} \cap \Supp(D + (s))_{v}^{\an} \not= \emptyset,
\]
which is a contradiction because $Y_{v}^{\an} \subseteq \Supp_{\mathrm{ess}}(\Prep(f))_v^{\an}$.
\end{proof}

%
%

\section{Examples}
\label{sec:examples}
In this section, we give several examples to apply Theorem~\ref{thm:canonical:comp}.

\begin{Example}[Abelian variety]
\label{exam:abelian:variety}
Let $A$ be an abelian variety over a number field $K$.
Let $D$ be an ample and symmetric $\RR$-Cartier divisor on $A$,
that is, there are ample and symmetric Cartier divisors $D_1, \ldots, D_r$ on $A$
and $a_1, \ldots, a_r \in \RR_{>0}$ with $D = a_1 D_1 + \cdots + a_r D_r$.
Then 
$[2]^*(D) = 4D + (\varphi)$ for some $\varphi \in \Rat(A)^{\times}_{\RR}$.
Let $\overline{D}$ be the canonical compactification of $D$ with respect to $[2]$.
Note that $\Prep([2]_{\sigma})$   is dense in $A_{\sigma}(\CC)$
with respect to the analytic topology
for $\sigma \in K(\CC)$.
Thus, by Lemma~\ref{lem:can:comp:nef} and Theorem~\ref{thm:canonical:comp},
$\overline{D}$ is nef and
$\overline{D}$ does not have the Dirichlet property.
\end{Example}

\begin{Example}[Projective line]
\label{exam:projective:line}
Let $E$ be an elliptic curve over a number field $K$,
$X := E/[\pm 1]$ and
$\rho : E \to X$ the natural morphism.
Note that $X \simeq \PP^1_K$ and
the endomorphism $[2] : E \to E$ descends to an endomorphism
$f : X \to X$, that is, the following diagram is commutative:
\[
\xymatrix{
E \ar[r]^{[2]} \ar[d]_{\rho} &  E \ar[d]^{\rho} \\
X \ar[r]^{f}  & X
}
\]
Clearly $\rho(\Prep([2])) \subseteq \Prep(f)$.
In particular, $\Prep([f]_{\sigma})$ is dense in $X_{\sigma}(\CC)$
with respect to the analytic topology
for $\sigma \in K(\CC)$.
Let $D$ be an ample Cartier divisor on $X$.
Note that $\rho^*(D)$ is symmetric because $\rho \circ [-1] = \rho$, 
so that there is $\varphi' \in \Rat(E)^{\times}$
with $[2]^*(\rho^*(D)) = 4 \rho^*(D) + (\varphi')$, that is,
$\rho^*(f^*(D) - 4D) =  (\varphi')$. Therefore,
if we set $\varphi = N(\varphi')^{1/2} \in \Rat(X)^{\times}_{\QQ}$, then
$f^*(D) = 4 D + (\varphi)$, where $N : \Rat(E)^{\times} \to \Rat(X)^{\times}$
is the norm map.
Let $\overline{D}$ be the canonical compactification of $D$ with respect to $f$.
By Lemma~\ref{lem:can:comp:nef} and Theorem~\ref{thm:canonical:comp},
$\overline{D}$ is nef and
the Dirichlet property of $\overline{D}$ does not hold.

\bigskip
Here let us consider a special elliptic curve $E$ due to Tate, that is,
\[
E := \Proj\left(K[X, Y, Z]/(Y^2Z + XYZ + \epsilon^2 YZ^2 - X^3)\right)
\]
where $\epsilon = (5 + \sqrt{29})/2$ and $K = \QQ(\epsilon)$.
It has a smooth model
\[
\EEE = \Proj\left( O_K[X, Y, Z]/(Y^2Z + XYZ + \epsilon^2 YZ^2 - X^3)\right)
\]
over $O_K := \ZZ[\epsilon]$.
Let $\EEE \dashrightarrow \PP^1_{O_K}$ be a rational map
induced by the homomorphism
$O_K[X, Z] \to O_K[X, Y, Z]/(Y^2Z + XYZ + \epsilon^2 YZ^2 - X^3)$, that is,
$\EEE \dashrightarrow \PP^1_{O_K}$ is the projection at $(0:1:0)$.
Note that $\EEE \dashrightarrow \PP^1_{O_K}$ actually extends to a morphism 
$\rho : \EEE \to \PP^1_{O_K}$ because the tangent line at $(0:1:0)$ is given by
$\{ Z = 0 \}$.

\begin{Claim}
There is a morphism $f : \PP^1_{O_K} \to \PP^1_{O_K}$ such that
the following diagram is commutative:
\[
\xymatrix{
\EEE \ar[r]^{[2]} \ar[d]_{\rho} &  \EEE \ar[d]^{\rho} \\
\PP^1_{O_K}\ar[r]^{f}  & \PP^1_{O_K}
}
\]
\end{Claim}

\begin{proof}
The $x$-coordinate of $[2](P)$ for $P = (x : y : 1) \in E$ is
given by
\[
\frac{x^4 - \epsilon^2x^2 -2 \epsilon^4 x}{4x^3 + x^2 + 2\epsilon^2 x + \epsilon^4}.
\]
Therefore, if we consider a rational map $f : \PP^1_{O_K} \dashrightarrow \PP^1_{O_K}$ given by
\[
f(x : z) := (x^4 - \epsilon^2x^2z^2 -2 \epsilon^4 xz^3 : 4x^3z + x^2z^2 + 2\epsilon^2 x z^3+ \epsilon^4z^4),
\]
then the diagram 
\[
\xymatrix{
\EEE \ar[r]^{[2]} \ar[d]_{\rho} &  \EEE \ar[d]^{\rho} \\
\PP^1_{O_K}\ar@{-->}[r]^{f}  & \PP^1_{O_K}
}
\]
is commutative as rational maps,
so that we need to see that $f$ extends to a morphism
$f : \PP^1_{O_K} \to \PP^1_{O_K}$.
Let $F$ be either $\overline{\QQ}$ or $\overline{\FF}_p$,
where $\FF_p = \ZZ/p\ZZ$ for a prime $p$.
It is sufficient to show that if $(x, z) \in F^2$
satisfies a system of equations
\[
\begin{cases}
x^4 - \epsilon^2x^2z^2 -2 \epsilon^4 xz^3 = 0, \\
4x^3z + x^2z^2 + 2\epsilon^2 x z^3+ \epsilon^4z^4 = 0,
\end{cases}
\]
then $x = z = 0$.
We assume the contrary, that is, the above has a solution
$(x, z) \in F^2 \setminus \{ (0, 0) \}$.
As $z \not= 0$, we may assume $z = 1$, so that $x \not= 0$, and hence
\[
\begin{cases}
x^3 - \epsilon^2x -2 \epsilon^4  = 0, \\
4x^3 + x^2 + 2\epsilon^2 x+ \epsilon^4 = 0.
\end{cases}
\]
Therefore,
$0 = (4x^3 + x^2 + 2\epsilon^2 x+ \epsilon^4) - 4(x^3 - \epsilon^2x -2 \epsilon^4) 
= (x + 3 \epsilon^2)^2$, that is, $x = -3\epsilon^2$.
Thus, as $(-3\epsilon^2)^3 - \epsilon^2(-3\epsilon^2) -2 \epsilon^4 = 0$ and $\epsilon \not= 0$,
we have $27 \epsilon^2 = 1$. On the other hand, since
$\epsilon^2 - 5 \epsilon - 1 = 0$, we obtain $27 \cdot 5 \epsilon = -26$, so that
$27 \cdot 25 = 27 \cdot 25 \cdot 27 \epsilon^2 = (27 \cdot 5 \epsilon)^2 = 26^2$
in $F$. Note that $26^2 - 27 \cdot 25 = 1$, and hence $1 = 0$ in $F$,
which is a contradiction. 
\end{proof}

Since the norm map $N : \Rat(\EEE)^{\times} \to \Rat(\PP_{O_K}^1)^{\times}$
is a homomorphism, we have the natural extension
\[
N_{\QQ} : \Rat(\EEE)^{\times}_{\QQ} := \Rat(\EEE)^{\times} \otimes \QQ \longrightarrow
\Rat(\PP_{O_K}^1)^{\times}_{\QQ} := \Rat(\PP_{O_K}^1)^{\times} \otimes \QQ.
\]
Let $\DDD$ be an ample Cartier divisor on $\PP^1_{O_K}$.
As the following diagram 
\[
\xymatrix{
\EEE \ar[rr]^{[-1]} \ar[dr]_{\rho} & & \EEE \ar[ld]^{\rho} \\
 & \PP^1_{O_K} &
}
\]
is commutative, $\rho^*(\DDD)$ is symmetric, 
so that $[2]^*(\rho^*(\DDD)) - 4 \rho^*(\DDD) - (\varphi')$ is vertical for some $\varphi' \in \Rat(\EEE)^{\times}$.
As the class group of $\QQ(\epsilon)$ is finite, there is $\lambda \in K^{\times}_{\QQ}$ such that
\[
\rho^*(f^*(\DDD) - 4 \DDD) = [2]^*(\rho^*(\DDD)) - 4 \rho^*(\DDD) = (\lambda\varphi'),
\]
and hence, if we set $\varphi = N_{\QQ}(\lambda \varphi')^{1/2} \in \Rat(\PP^1_{O_K})^{\times}_{\QQ}$,
then $f^*(\DDD) = 4 \DDD + (\varphi)$.
Let $g_{\infty}$ be an $F_{\infty}$-invariant $\DDD$-Green function of $(C^0 \cap \Tpsh)$-type
on $\PP^1_{O_K}(\CC)$ such that $f^*(g) = 4 g - \log \vert \varphi \vert^2_{\infty}$.
By Lemma~\ref{lem:can:comp:nef} and \cite[Proposition~2.1.7]{MoAdel},
an arithmetic Cartier divisor $\overline{\DDD} := (\DDD, g_{\infty})$ on $\PP^1_{O_K}$ is nef
and, by Theorem~\ref{thm:canonical:comp},
$\overline{\DDD} + \widehat{(\psi)}$ is not effective for all
$\psi \in \Rat(\PP^1_{O_K})^{\times}_{\RR}$.
Further, $\rho^*(\overline{\DDD})$ is nef and
$\rho^*(\overline{\DDD}) + \widehat{(\phi)}$ is not effective for all
$\phi \in \Rat(\AAA)^{\times}_{\RR}$.
\end{Example}

\begin{Example}
\label{exam:converse:thm}
Here let us give an example due to Burgos i Gil,
which shows that the converse of Theorem~\ref{thm:main:result} in the introduction 
does not hold in general.

Let $E$ be an elliptic curve over $\QQ$ and $\PP^1_{\QQ} := \Proj(\QQ[x, y])$.
Let $D_1$ (resp. $D_2$) be the Cartier divisor on $E$ (resp. $\PP^1_{\QQ}$) given by
the zero point (resp. $\{ x = 0 \}$).
Then there is $\varphi \in \Rat(E)^{\times}$ with $[2]^*(D_1) = 4 D_1 + (\varphi)$.
Let $h : \PP^1_{\QQ} \to \PP^1_{\QQ}$ be the endomorphism given by $(x : y) \mapsto (x^4 : y^4)$.
Then $h^*(D_2) = 4 D_2$.
We set 
\[
X := E \times \PP^1_{\QQ},\quad 
f := [2] \times h : X \to X\quad\text{and}\quad
D := p_1^*(D_1) + p_2^*(D_2),
\]
where $p_1 : X \to E$ and $p_2 : X \to \PP^1_{\QQ}$ are
the projections to $E$ and $\PP^1_{\QQ}$, respectively.
As the following diagrams are commutative,
\[
\begin{CD}
X @>{f}>> X \\
@V{p_1}VV @VV{p_1}V \\
E @>>{[2]}> E
\end{CD}
\qquad\qquad
\begin{CD}
X @>{f}>> X \\
@V{p_2}VV @VV{p_2}V \\
\PP^1_{\QQ} @>>{h}> \PP^1_{\QQ}
\end{CD}
\]
we have
\begin{align*}
f^*(D) & = f^*(p_1^*(D_1)) + f^*(p_2^*(D_2)) = p_1^*([2]^*(D_1)) + p_2^*(h^*(D_2)) \\
& = p_1^*(4D_1 + (\varphi)) + p_2^*(4D_2) = 4 D + (p_1^*(\varphi)).
\end{align*}
Let $\overline{D}$ be the canonical compactification of $D$ with respect to $f$.

\begin{Claim}
\begin{enumerate}
\renewcommand{\labelenumi}{(\arabic{enumi})}
\item
$\overline{D}$ is nef.

\item
$\Prep(f_{v})^{\an}$ is not dense in $X^{\an}_{v}$ with respect to
the analytic topology for all $v \in M_{\QQ} \cup \{ \infty \}$,
where $\infty$ is the unique embedding $\QQ \hookrightarrow \CC$.

\item
$\Supp_{\mathrm{ess}}(\Prep(f))^{\an}_{\infty} = E(\CC) \times S^1$,
where $S^{1}= \{ (\zeta : 1) \in \PP^{1}(\CC) \mid | \zeta | = 1 \}$.

\item
The Dirichlet property of $\overline{D}$ does not hold.
\end{enumerate}
\end{Claim}

\begin{proof}
(1) follows from Lemma~\ref{lem:can:comp:nef}.

(2) As $(p_2)_{v}^{\an} : X_{v}^{\an} \to (\PP^1_{\QQ})^{\an}_{v}$ is surjective and
$(p_2)_{v}^{\an}(\Prep(f_{v})^{\an}) \subseteq \Prep(h_{v})^{\an}$,
it is sufficient to show that $\Prep(h_{v})^{\an}$ is not dense in $(\PP^1_{\QQ})^{\an}_{v}$.
Note that
\[
\Prep(h) = \{ (0:1), (1:0) \} \cup \{ (\zeta, 1) \mid \text{$\zeta \in \overline{\QQ}$ and $\zeta^m = 1$ for some $m \in \ZZ_{>0}$} \},
\]
so that the assertion is obvious if $v = \infty$. We assume that $v = p$ for some prime $p$.
Let $w \in \Prep(h_{p})^{\an} \cap U^{\an}_{\QQ_p}$,
where $U$ is the Zariski open set of $\PP^1_{\QQ}$ given by
$U := \{ x \not= 0,\ y \not= 0 \}$ and
$U_{\QQ_p} := U \times_{\Spec(\QQ)} \Spec(\QQ_p)$.
In the same way as Proposition~\ref{prop:thm:canonical:comp},
there is $\xi \in \Prep(h)$ such that $w$ is one of valuations arising from $\xi$, that is, if we set
\[
\QQ(\xi) \otimes \QQ_p = K_1 \oplus \cdots \oplus K_r \quad (\text{the sum of finite extension fields over $\QQ_p$}),
\]
then $w$ is the valuation of some $K_i$. 
Put $z := X/Y$.
As $z(\xi)^m = 1$ for some $m \in \ZZ_{>0}$, we obtain $z^m = 1$ at $K_i$, so that
$\vert z \vert_{w} = 1$.
Therefore we have
\[
\Prep(h_{p})^{\an} \cap U^{\an}_{\QQ_p} \subseteq \{ w \in U^{\an}_{\QQ_p} \mid \vert z \vert_{w} = 1 \},
\]
and hence $\Prep(h_{p})^{\an}$ is not dense.

(3) We need to see 
$\overline{\Prep(f) \setminus Y(\CC)} = E(\CC) \times S^{1}$
for any proper subscheme $Y$ of $E \times \PP^1_{\QQ}$.
Note that
\[
\overline{\Prep(f)} = E(\CC) \times S^{1}
\quad\text{and}\quad
\overline{\Prep(f)} \setminus Y(\CC) \subseteq
\overline{\Prep(f) \setminus Y(\CC)},
\]
so that it is sufficient to check
$E(\CC) \times S^1 \subseteq \overline{(E(\CC) \times S^{1}) \setminus Y(\CC)}$.

We set
$T = \{ \zeta \in S^1 \mid E(\CC) \times \{ \zeta \} \subseteq Y(\CC) \}$.
Let us see that $T$ is finite. Otherwise,
as 
$E(\CC) \times T \subseteq Y(\CC)$
and $E(\CC) \times T$ is Zariski dense in $E(\CC) \times \PP^1(\CC)$,
we have $Y(\CC) = E(\CC) \times \PP^1(\CC)$, which is
a contradiction.

Since $\overline{(E(\CC) \times \{ \zeta \}) \setminus Y(\CC)}
= E(\CC) \times \{ \zeta \}$ for $\zeta \in S^1 \setminus T$,
we obtain
\[
E(\CC) \times (S^1 \setminus T) \subseteq \overline{E(\CC) \times S^1 \setminus Y(\CC)}.
\]
Thus the assertion follows because
$\overline{S^1 \setminus T} = S^1$.

(4)
follows from (2) and Theorem~\ref{thm:canonical:comp}.
\end{proof}
\end{Example}

\renewcommand{\theTheorem}{\arabic{section}.\arabic{subsection}.\arabic{Theorem}}
\renewcommand{\theClaim}{\arabic{section}.\arabic{subsection}.\arabic{Theorem}.\arabic{Claim}}

\section{
Measure-theoretical approach to the study of Dirichlet property}
\label{sec:non:dense}
The purpose of the section is to study the Dirichlet property in a functional point of view. Our method consists of introducing some (possible non-linear) functionals on the spaces of continuous functions on the analytic fibers of the arithmetic variety. The Dirichlet property leads to conditions on the supports of these functionals.

We fix a geometrically integral  projective scheme $X$ of dimension $d$ defined  over a number field $K$ and denote by $\pi:X\rightarrow\Spec K$ the structural morphism. In the following subsection, we will establish an abstract framework to study the consequences of the Dirichlet property by the functional approach. We then specify the theorem for different choices of the functionals, notably those coming from the asymptotic maximal slope and the volume function. 

\subsection{A formal functional analysis on Dirichlet property}\label{Sub:formalDirichlet}
\setcounter{Theorem}{0}

Let $V$ be a vector subspace of $\widehat{\Div}_{\mathbb R}(X)$ containing all principal divisors and 
let $V_+$ denote the subset of all effective adelic arithmetic Cartier divisors in $V$. Let $C_\circ$ be a subset of $V$ verifying the following conditions~:
\begin{enumerate}[(a)]
\item for any $\overline D\in C_\circ$ and 
$\lambda>0$, one has $\lambda\overline D\in C_\circ$;
\item for any $\overline D_0\in C_\circ$ and  
$\overline D\in V_+$, there exists $\varepsilon>0$ such that $\overline D_0+\varepsilon\overline D\in C_\circ$ for any $\varepsilon\in\mathbb R$ with $0\leq \varepsilon\leq\varepsilon_0$;
\item for any $\overline D\in C_\circ$ and 
$\phi\in\Rat(X)_{\mathbb R}^{\times}$, one has $\overline D+\widehat{(\phi)}\in C_\circ$.
\end{enumerate}
In other terms, $C_\circ$ is a cone in $V$ which is open in the directions in $V_+$ and invariant 
under translations by a principal divisor.

Assume given a map
$\mu : C_{\circ} \to \RR$ which verifies the following properties~:
\begin{enumerate}[(1)]
\item there exists a positive number $a$ such that $\mu(t\overline D)=t^a\mu(\overline D)$ for all adelic arithmetic $\mathbb R$-Cartier divisor $\overline D\in C_0$ and $t>0$;
\item 
for any $\overline D\in C_\circ$ and  
$\phi\in\Rat(X)^\times_{\mathbb R}$, one has  $\mu(\overline D+\widehat{(\phi)})=\mu(\overline D)$. 
\end{enumerate}

For $\overline{D} \in C_{\circ}$ and $\overline{E} \in V_+$, 
we define $\nabla^+_{\overline{E}}\mu(\overline{D})$ to be
\[
\nabla^+_{\overline{E}}\mu(\overline{D}) =
\limsup_{\epsilon \to 0+} \frac{\mu(\overline{D} + \epsilon \overline{E}) - \mu(\overline{D})}{\epsilon},
\]
which might be $\pm\infty$. Note that, for any $\overline D\in C_\circ$, the function $\overline E\mapsto\nabla_{\overline E}^+\mu(\overline D)$ is positively homogeneous. Moreover, for any $\overline E\in V$ and 
$t>0$, one has $\nabla_{\overline E}^+\mu(t\overline D)=t^{a-1}\nabla_{\overline E}^+(\overline D)$.   In addition to (1) and (2), assume the following property: 
\begin{enumerate}[(1)]
\setcounter{enumi}{2}
\item there exists a map $\nabla_\mu:\widehat{\Div}_{\mathbb R}(X)_+\times C_\circ\rightarrow\mathbb R\cup\{\pm\infty\}$ such that \[\nabla_{\mu}(\overline E,\overline D)=\nabla_{\overline E}^+\mu(\overline D)\quad \text{for $\overline E\in V_+$ and $\overline D\in C_\circ$},\] where $\widehat{\Div}_{\mathbb R}(X)_+$ denotes the set of all effective adelic arithmetic $\mathbb R$-Cartier divisors.
\end{enumerate}

Denote by $C_{\circ\circ}$ the subset of $C_{\circ}$ of adelic arithmetic $\mathbb R$-Cartier divisor $\overline D$ such that, the map $\overline E\mapsto\nabla_\mu(\overline E,\overline D)$ preserves the order, namely, for any couple $(\overline E_1,\overline E_2)$ of  elements in $\widehat{\Div}_{\mathbb R}(X)_+$ such that $\overline E_1\leq \overline E_2$, one has $\nabla_{\mu}(\overline E_1,\overline D)\leq\nabla_{\mu}(\overline E_2,\overline D)$. 
If $\overline D$ is an element in $C_{\circ\circ}$, then the map $\nabla_{\mu}$ defines, for any $v\in M_K\cup K(\mathbb C)$, a non-necessarily additive functional
\[\Psi_{\overline D,v}^{\mu}:C^0(X_v^{\mathrm{an}})_+\longrightarrow [0,+\infty],\quad\Psi_{\overline D,v}^{\mu}(f_v):=\nabla_\mu({\overline{O}(f_v)},\overline{D}),\]
where $C^0(X_v^{\mathrm{an}})_+$ denotes the cone of non-negative continuous functions on $X_v^{\mathrm{an}}$.

\begin{Definition}
We define the \emph{support} of $\Psi_{\overline D,v}^{\mu}$ to be the set $\Supp(\Psi_{\overline D,v}^{\mu})$ of all $x\in X_v^{\mathrm{an}}$ such that $\Psi_{\overline D,v}^{\mu}(f_v)>0$ for any non-negative continuous function $f_v$ on $X_v^{\mathrm{an}}$ verifying $f_v(x)>0$. Note that $F_{\infty}(\Supp(\Psi^{\mu}_{\overline{D},\sigma})) = \Supp(\Psi^{\mu}_{\overline{D},\bar{\sigma}})$
for $\sigma \in K(\CC)$ because $\Psi^{\mu}_{\overline{D},\sigma}(f_{\sigma}) =
\Psi^{\mu}_{\overline{D},\bar{\sigma}}(F_{\infty}^*(f_{\sigma}))$ for
$f_{\sigma} \in C^0(X^{\an}_{\sigma})_{+}$.
\end{Definition}

\begin{Theorem}\label{Thm:consequenceDirichlet}
Let $\overline{D}$ be an element of  $C_{\circ\circ}$ with $\mu(\overline{D}) = 0$.  If $s$ is an element of $\Rat(X)_{\mathbb R}^\times$ with $\overline D+\widehat{(s)}\geq 0$, then 
\[
\Supp(\Psi_{\overline D,v}^{\mu})\cap \{ x \in X_v^{\an} \mid | s |_{g_v} < 1 \} = \varnothing
\]
for any $v\in M_K\cup K(\mathbb C)$.
\end{Theorem}
\begin{proof} 
We set $\overline D{}'=\overline D+\widehat{(s)} = (D', g')$ and $f_v = \min \{ g'_v, 1 \}$.
Thus, as 
\[
0 \leq \overline{O}(f_v) \leq \overline{D}{}'
\]
and $\overline{D} \in C_{\circ\circ}$, one has
\[
0 =  \nabla_\mu((0,0),\overline{D})\leq \Psi_{\overline{D},v}(f_v) = \nabla_\mu({\overline{O}(f_v)},\overline{D}) \leq \nabla_\mu(\overline D{}',\overline D)=\nabla^+_{\overline{D}{}'}\mu(\overline{D}).
\]
On the other hand, by using the properties (1) and (2), one obtains
\[
\mu(\overline{D} + \epsilon \overline{D}{}') - \mu(\overline{D}) = \mu(\overline{D} + \epsilon \overline{D}) - \mu(\overline{D}) = ((1+\epsilon)^a-1)\mu(\overline{D}),
\]
and hence $\nabla^+_{\overline{D}{}'}\mu(\overline{D}) = a \mu(\overline{D}) = 0$.
Therefore, $\Psi_{\overline{D},v}(f_v) = 0$, so that
\[
\Supp(\Psi_{\overline{D}, v}) \cap \{ x \in X_v^{\an} \mid f_v(x) > 0 \} = \varnothing.
\]
Note that $g'_v = -\log | s |^2_{g_v}$. Thus, we can see that
\[
 \{ x \in X_v^{\an} \mid f_v(x) > 0 \} = \{ x \in X_v^{\an} \mid | s |_{g_v} < 1 \},
\]
as required.
\end{proof}

Under the assumptions of Theorem \ref{Thm:consequenceDirichlet}, we have the following corollaries.

\begin{Corollary}
Assume that the Dirichlet property holds for $\overline D$ and that $D$ is big.
For any $v\in M_K\cup K(\mathbb C)$, if 
$Y_v$
is a closed subvariety of $X_v$ of dimension $\geq 1$ 
such that 
$Y_v^{\mathrm{an}} \subseteq \Supp(\Psi_{\overline D,v}^{\mu})$, 
then 
$Y_v$ 
is contained in the the augmented base locus of $D_v$.
\end{Corollary}
\begin{proof} Let $s$ be an element of $\Rat(X)_{\mathbb R}^\times$ with $\overline D+\widehat{(s)}\geq 0$. We introduce $\overline D{}'=\overline D+\widehat{(s)}$ as in the proof of the theorem.
Assume that 
$Y_v$
is a closed subvariety of $X_v$ which is not contained in the augmented base locus of $D_v$ (which identifies with the augmented base locus of $D'_v$). Then the restriction of $D'_v$ on 
$Y_v$
is a big $\mathbb R$-Cartier divisor since the restricted volume of $D'_v$ on 
$Y_v$ 
is $>0$ (cf. \cite{ELMNP09}). Hence $[D'_v]$ has non-empty intersection with $Y_v$, which implies that 
$[D'_v]^{\mathrm{an}}\cap Y_v^{\mathrm{an}}\neq\varnothing$. 
Therefore, by the previous theorem, $Y_v^{\mathrm{an}}$ cannot be contained in the support of the functional 
$\Psi_{\overline D,v}^{\mu}$.  
\end{proof}

\begin{Corollary}
Assume that $(D \cdot C) > 0$ for any curve $C$ on $X$.
If the Dirichlet property holds for $\overline D$, then, for any $v \in M_K \cup K(\CC)$,
there is no subvariety 
$Y_v$
of $X_v$ such that 
$\dim Y_v \geq 1$ 
and 
$Y_v^{\an} \subseteq \Supp(\Psi_{\overline D,v}^{\mu})$.
\end{Corollary}

\begin{proof}
It is sufficient to prove that 
$(D_{v} \cdot C_v) > 0$ 
for any curve 
$C_v$ 
on $X_v$.
Indeed, there are a variety $W$ over $K$ and a subscheme $\CCC$ of $X \times W$ such that
$\Rat(W)$ is a subfield of $K_v$,
$\CCC$ is flat over $W$ and 
$\CCC \times_{W} \Spec(K_v) = C_v$.
Let $p$ and $q$ be the projections $X \times W \to W$ and $X \times W \to X$, respectively.
By our assumption, $(q^*(D) \cdot \CCC \cap p^{-1}(w)) > 0$ for any closed point $w$ of $W$,
so that $(q^*(D) \cdot \CCC_{\eta}) > 0$, where $\eta$ is the generic point of $W$, as required.
\end{proof}

\subsection{Asymptotic maximal slope}\label{Subsec:asypmaxs}
\newcommand{\sbullet}{{\scriptscriptstyle\bullet}}
\setcounter{Theorem}{0}
In this subsection, let $V=\widehat{\mathrm{Div}}_{\mathbb R}(X)$ and  $C_\circ$ be the cone of all adelic arithmetic $\mathbb R$-Cartier divisors $\overline D$ such that $D$ is big.

Let $\overline D$ be an adelic arithmetic $\mathbb R$-Cartier divisor in $C_\circ$ and $\zeta$ be an adelic arithmetic $\RR$-Cartier divisor on $\Spec K$ with $\adeg(\zeta) = 1$, we define $\widehat{\mu}_{\max}^{\mathrm{asy},\zeta}(\overline D)$ as
\[\sup \{ t \in \RR \mid \text{$\overline{D} - t \pi^*(\zeta)$ has the Dirichlet property}\}.\]
Note that for sufficiently negative number $t$, the adelic arithmetic $\mathbb R$-Cartier divisor $\overline D-t\pi^*(\zeta)$ is big (since $D$ is a big $\mathbb R$-divisor) and therefore has the Dirichlet property. Moreover, one has $\widehat{\mu}_{\max}^{\mathrm{asy},\zeta}(\overline D)\leq \widehat{\mu}_{\mathrm{ess}}(\overline D)$ (see Conventions and terminology \ref{CV:height}). Therefore $\widehat{\mu}_{\max}^{\mathrm{asy},\zeta}(.)$ is a real-valued function on $C_\circ$. By definition, for any $t\geq 0$ and any $\overline D\in C_\circ$ one has $\widehat{\mu}_{\max}^{\mathrm{asy},\zeta}(t\overline D)=t\widehat{\mu}_{\max}^{\mathrm{asy},\zeta}(\overline D)$.

The function $\widehat{\mu}_{\max}^{\mathrm{asy},\zeta}(\overline{D})$ is actually independent of
the choice of $\zeta$. This is a consequence of the following proposition. 

\begin{Proposition}
Let $\zeta_1$ and $\zeta_2$ be adelic arithmetic $\RR$-Cartier divisors on $\Spec K$.
Then $\adeg(\zeta_1) = \adeg(\zeta_2)$ if and only if
$\zeta_1 = \zeta_2 + \widehat{(\varphi)}$ for some 
$\varphi \in K^{\times}_{\RR} (:= K^{\times} \otimes \RR)$.
\end{Proposition}

\begin{proof}
It is sufficient to show that if $\adeg(\zeta) = 0$, then $\zeta = \widehat{(\varphi)}$ for some
$\varphi \in K^{\times}_{\RR}$.
We set $\zeta = \sum_{\mathfrak p \in M_K} a_{\mathfrak p} [\mathfrak p] + \sum_{\sigma \in K(\CC)} a_{\sigma} [\sigma]$.
As the class group of $K$ is finite, we may assume that $a_{\mathfrak p} = 0$ for all $\mathfrak p \in M_K$.
Therefore, Dirichlet's unit theorem implies the assertion.
\end{proof}

We shall use the expression $\widehat{\mu}_{\max}^{\mathrm{asy}}(.)$ to denote this function. By definition, for any adelic arithmetic $\mathbb R$-Cartier divisor $\zeta$ one has \[\widehat{\mu}_{\max}^{\mathrm{asy}}(\overline D+\pi^*(\zeta))=\widehat{\mu}_{\max}^{\mathrm{asy}}(\overline D)+\widehat{\deg}(\zeta).\]
This function has been introduced in the adelic line bundle setting in \cite{Chen10} in an equivalent form by using arithmetic graded linear series. We refer the readers to \S\ref{Subsec:comparison} \emph{infra} for more details.

If $\overline D$ is an adelic arithmetic $\mathbb R$-Cartier divisor verifying the Dirichlet property, then for any $\varphi\in\Rat(X)_{\mathbb R}^{\times}$, $\overline D+\widehat{(\varphi)}$ also verifies the Dirichlet property. Moreover, for any $\overline{D}{}'\geq\overline D$, the adelic arithmetic $\mathbb R$-Cartier divisor $\overline D{}'$ verifies the Dirichlet property. We deduce from these facts the following properties of the function $\widehat{\mu}_{\max}^{\mathrm{asy}}(.)$.

\begin{Proposition}
\begin{enumerate}[(1)]
\item Let $\overline D$ be an adelic arithmetic $\mathbb R$-Cartier divisor in $C_\circ$. For any $\varphi\in\Rat(X)_{\mathbb R}^{\times}$ one has $\widehat{\mu}_{\max}^{\mathrm{asy}}(\overline D+\widehat{(\varphi)})=\widehat{\mu}_{\max}^{\mathrm{asy}}(\overline D)$.
\item The function $\widehat{\mu}_{\max}^{\mathrm{asy}}(.)$ preserves the order relation, namely for $\overline D_1\leq\overline D_2$ in $C_\circ$ one has $\widehat{\mu}_{\max}^{\mathrm{asy}}(\overline D_1)\leq\widehat{\mu}_{\max}^{\mathrm{asy}}(\overline D_2)$.
\item The function $\widehat{\mu}_{\max}^{\mathrm{asy}}(.)$ is super-additive, namely \[\widehat{\mu}_{\max}^{\mathrm{asy}}(\overline D_1+\overline D_2)\geq \widehat{\mu}_{\max}^{\mathrm{asy}}(\overline D_1)+\widehat{\mu}_{\max}^{\mathrm{asy}}(\overline D_2)\] for $\overline D_1$ and $\overline D_2$ in $C_\circ$. 
\end{enumerate} 
\end{Proposition}

The Theorem \ref{Thm:consequenceDirichlet} leads immediately to the following corollary.

\begin{Corollary}\label{Cor:muasy}
Let $\overline D$ be an adelic arithmetic $\mathbb R$-Cartier divisor such that $D$ is big and that $\widehat{\mu}_{\max}^{\mathrm{asy}}(\overline D)=0$. If $s$ is an element of $\Rat(X)_{\mathbb R}^\times$ with $\overline D+\widehat{(s)}\geq 0$, then 
\[
\Supp(\Psi_{\overline D,v}^{\widehat{\mu}_{\max}^{\mathrm{asy}}})\cap \{ x \in X_v^{\an} \mid | s |_{g_v} < 1 \} = \varnothing
\]
for any $v\in M_K\cup K(\mathbb C)$.
\end{Corollary}

The function $\widehat{\mu}^{\mathrm{asy}}_{\max}(.)$ is important in the study of Dirichlet's theorem. In fact, it is not only the threshold of the Dirichlet property but also the pseudo-effectivity.

\begin{Lemma}\label{Lem:continuity}
Let $(\overline D_i)_{i=1}^n$ be a family of  adelic arithmetic $\mathbb R$-Cartier divisors
on $X$, and $\overline D$ be an adelic arithmetic $\mathbb R$-Cartier divisor on $X$ such that $D$ is big. Then one has
\begin{equation*}\lim_{|\boldsymbol{t}|\rightarrow 0}\widehat{\mu}_{\max}^{\mathrm{asy}}(\overline D+t_1\overline D_1+\cdots+t_n\overline D_n)=\widehat{\mu}_{\max}^{\mathrm{asy}}(\overline D),\end{equation*}
where for $\boldsymbol{t}=(t_1,\ldots,t_n)\in\mathbb R^n$, the expression $|\boldsymbol{t}|$ denotes $\max\{|t_1|,\ldots,|t_n|\}$.
\end{Lemma}
\begin{proof}
If we replace $\overline D$ by $\overline D+\pi^*(\zeta)$, where $\zeta$ is an adelic arithmetic Cartier $\mathbb R$-divisor on $\Spec K$, both sides of the equality to be proved differ the initial value by $\widehat{\deg}(\zeta)$. Hence one may assume that $\overline D$ is big. Moreover,
without loss of generality, one may assume that $(\overline D_i)_{i=1}^n$ are adelic arithmetic Cartier divisors which are big 
and effective
(by possibly augmenting the number of adelic arithmetic Cartier divisors in the family). In fact, 
each $\overline D_i$ is $\RR$-linearly equivalent to
an $\mathbb R$-linear combination of big and effective arithmetic Cartier divisors. Then by using the fact that the function $\widehat{\mu}_{\max}^{\mathrm{asy}}(.)$ preserves the order, one obtains
\[\begin{split}\widehat{\mu}_{\max}^{\mathrm{asy}}(\overline D-|\boldsymbol{t}|(\overline D_1+\cdots+\overline D_n))&\leq \widehat{\mu}_{\max}^{\mathrm{asy}}(\overline D+t_1\overline D_1+\cdots+t_n\overline D_n)\\&\leq \widehat{\mu}_{\max}^{\mathrm{asy}}(\overline D+|\boldsymbol{t}|(\overline D_1+\cdots+\overline D_n)).
\end{split}\]
Therefore we have reduced the problem to the case where $n=1$ and $\overline D_1$ is a big and effective adelic arithmetic Cartier divisor. Let $a>0$ be a real number such that $a\overline D-\overline D_1$ and $a\overline D+\overline D_1$ are both 
$\RR$-linearly equivalent to effective adelic $\RR$-Cartier divisors.
By the positive homogenity of the function $\widehat{\mu}_{\max}^{\mathrm{asy}}(.)$, for any $t>0$  one has
\[\widehat{\mu}_{\max}^{\mathrm{asy}}(\overline D)\geq \widehat{\mu}_{\max}^{\mathrm{asy}}(\overline D-t\overline D_1)\geq(1-at)\widehat{\mu}_{\max}^{\mathrm{asy}}(\overline D)\]
Hence $\displaystyle\lim_{t\rightarrow0+}\widehat{\mu}_{\max}^{\mathrm{asy}}(\overline D-t\overline D_1)=\widehat{\mu}_{\max}^{\mathrm{asy}}(\overline D)$. Similarly, one has $\displaystyle\lim_{t\rightarrow0+}\widehat{\mu}_{\max}^{\mathrm{asy}}(\overline D+t\overline D_1)=\widehat{\mu}_{\max}^{\mathrm{asy}}(\overline D)$. The result is thus proved.
\end{proof}

\begin{Proposition}[%
cf. Proposition~\ref{Pro:cripseudoeff}
]\label{Pro:big:cripseudoeff}
Let $\overline D$ be an adelic arithmetic $\mathbb R$-Cartier divisor such that $D$ is big. Then $\overline D$ is big (resp. pseudo-effective) if and only if $\widehat{\mu}_{\max}^{\mathrm{asy}}(\overline D)>0$ (resp. $\widehat{\mu}_{\max}^{\mathrm{asy}}(\overline D)\geq 0$).
\end{Proposition}
\begin{proof}
Assume that $\overline D$ is big. 
Let $\zeta$ be an arithmetic $\mathbb R$-Cartier divisor on $\Spec K$ such that $\widehat{\deg}(\zeta)>0$. Since $\overline D$ is big, for sufficiently small $t>0$, the adelic arithmetic $\mathbb R$-Cartier divisor $\overline D-t\pi^*(\zeta)$ is big, and hence verifies the Dirichlet property. Therefore one has \[\widehat{\mu}_{\max}^{\mathrm{asy}}(\overline D)=t\widehat{\deg}(\zeta)+\widehat{\mu}_{\max}^{\mathrm{asy}}(\overline D-t\pi^*(\zeta))\geq t\widehat{\deg}(\zeta)>0.\]

Conversely, assume that 
$\widehat{\mu}_{\max}^{\mathrm{asy}}(\overline D)>0$. We write $\overline D$ as an $\mathbb R$-linear combination 
\[\overline D=a_1\overline D_1+\cdots+a_n\overline D_n,\]
where $(\overline D_i)_{i=1}^n$ are big adelic arithmetic $\mathbb Q$-Cartier divisors. For any $\varepsilon>0$, we can choose $b_1,\ldots,b_n$ in $\mathbb Q$ such that $a_i-\varepsilon\leq b_i< a_i$. Then $\overline D_\varepsilon=b_1\overline D_1+\cdots+b_n\overline D_n$ is an adelic arithmetic $\mathbb Q$-Cartier divisor and $D-D_\varepsilon$ is big. Moreover, if $\varepsilon$ is sufficiently small, $D_\varepsilon$ is big and $\widehat{\mu}_{\max}^{\mathrm{asy}}(\overline D_\varepsilon)>0$. By \cite[Proposition 3.11]{Chen10b} (see also Proposition \ref{Pro:equivalenceform}), one obtains that $\overline D_{\varepsilon}$ is big. Therefore $\overline D$ is also big.

Assume that $\overline D$ is pseudo-effective. Let $\overline D{}'$ be a big adelic arithmetic $\mathbb R$-Cartier divisor. For any $t>0$, $t\overline D{}'+\overline D$ is big. Therefore $\widehat{\mu}_{\max}^{\mathrm{asy}}(t\overline D{}'+\overline D)>0$. By the continuity of the function $\widehat{\mu}_{\max}^{\mathrm{asy}}(.)$ (see Lemma \ref{Lem:continuity}), one obtains that $\widehat{\mu}_{\max}^{\mathrm{asy}}(\overline D)\geq 0$. 

Conversely, assume that $\widehat{\mu}_{\max}^{\mathrm{asy}}(\overline D)\geq 0$. If $\overline D{}'$ is a big adelic arithmetic $\mathbb R$-Cartier divisor, then $D+D'$ is big since $D$ is pseudo-effective and $D'$ is big. Moreover, one has \[\widehat{\mu}_{\max}^{\mathrm{asy}}(\overline D+\overline D{}')\geq \widehat{\mu}_{\max}^{\mathrm{asy}}(\overline D)+\widehat{\mu}_{\max}^{\mathrm{asy}}(\overline D{}')>0.\]
Hence $\overline D+\overline D{}'$ is big. Therefore $\overline D$ is pseudo-effective. 
\end{proof}

\begin{Remark}
Let $\overline D$ be an adelic arithmetic $\mathbb R$-Cartier divisor on $X$ such that $D$ is big. The above proposition shows that
\[\widehat{\mu}_{\max}^{\mathrm{asy}}(\overline D)=\max\{t\in\mathbb R\,|\,\overline D-t\pi^*(\zeta)\text{ is pseudo-effective}\},\]
where $\zeta$ is any adelic arithmetic $\mathbb R$-Cartier divisor such that $\widehat{\deg}(\zeta)=1$. Note that big adelic arithmetic $\mathbb R$-Cartier divisors verify the Dirichlet property. Therefore, in order to construct counter-examples to the Dirichlet property, one should examine adelic arithmetic $\mathbb R$-Cartier divisor of the form $\overline D-\widehat{\mu}_{\max}^{\mathrm{asy}}(\overline D)\pi^*(\zeta)$, where $\widehat{\deg}(\zeta)=1$. Note that the functionals $\Psi_{\overline D,v}^{\widehat{\mu}_{\max}^{\mathrm{asy}}}$ remain invariant if one replaces $\overline D$ by a translation of $\overline D$ by the pull-back of an adelic arithmetic $\mathbb R$-Cartier divisor on $\Spec K$. Therefore the study of these functionals will very possibly provide a large family of counter examples to the Dirichlet property and suggest a way to characterize it. 
\end{Remark}



\subsection{Volume function}\label{Subsec:volfuc}
\setcounter{Theorem}{0}
In this subsection, we still assume that $V=\widehat{\mathrm{Div}}_{\mathbb R}(X)$ and  $C_\circ$ is the cone of all adelic arithmetic $\mathbb R$-Cartier divisors $\overline D$ such that $D$ is big.

Let $\mu$ be the arithmetic volume function $\widehat{\vol}$ (see Conventions and terminology \ref{CT:effectivesection}). Note that one has $\widehat{\vol}(t\overline D)=t^{d+1}\widehat{\vol}(\overline D)$. Moreover, for any $\phi\in\Rat(X)_{\mathbb R}^{\times}$, one has $\widehat{\vol}(\overline D+\widehat{(\phi)})=\widehat{\vol}(\overline D)$. Therefore the function $\mu=\widehat{\vol}(.)$ verifies the conditions (1)--(3) in \S\ref{Sub:formalDirichlet}. Moreover, the volume function preserves the order relation. Namely for $\overline D_1\leq\overline D_2$ one has $\widehat{\vol}(\overline D_1)\leq\widehat{\vol}(\overline D_2)$.

A direct consequence of Theorem \ref{Thm:consequenceDirichlet} is the following corollary.
\begin{Corollary}\label{Cor:volder}
Let $\overline D$ be an adelic arithmetic $\mathbb R$-Cartier divisor such that $D$ is big and that $\widehat{\vol}(\overline D)=0$. If $s$ is an element of $\Rat(X)_{\mathbb R}^\times$ with $\overline D+\widehat{(s)}\geq 0$, then 
\[
\Supp(\Psi_{\overline D,v}^{\widehat{\vol}})\cap \{ x \in X_v^{\an} \mid | s |_{g_v} < 1 \} = \varnothing
\]
for any $v\in M_K\cup K(\mathbb C)$.
\end{Corollary}

\subsection{Self-intersection number}
\setcounter{Theorem}{0}
In this subsection, let $V$ be the subspace of $\widehat{\Div}_{\mathbb R}(X)$ consisting of integrable adelic arithmetic $\mathbb R$-Cartier divisors. Let $C_\circ=V$. We define the function $\mu:C_\circ\rightarrow\mathbb R$ as $\mu(\overline D):=\widehat{\deg}(\overline D{}^{d+1})$. The function $\mu$ verifies the conditions (1) and (2) of \S\ref{Sub:formalDirichlet}. For $\overline D\in C_\circ$ and $\overline E\in V_+$, we define $\nabla_{\overline E}^+\mu(\overline D)$ as follows
\[\nabla_{\overline E}^+\mu(\overline D)=\lim_{\varepsilon\rightarrow 0+}\frac{\mu(\overline D+\varepsilon\overline E)-\mu(\overline D)}{\varepsilon}=(d+1)\widehat{\deg}(\overline D{}^d\cdot\overline E).\]
The function extends naturally to the whole space $\widehat{\Div}_{\mathbb R}(X)$ of adelic arithmetic $\mathbb R$-Cartier divisors (see Conventions and terminology \ref{CV:measure}) and thus defines a map $\nabla_\mu:\widehat{\Div}_{\mathbb R}(X)\times C_\circ\rightarrow\mathbb R$. Note that the subset $C_{\circ\circ}$ of $\overline D\in C_\circ$ such that $\nabla_{\mu}(.,\overline D)$ preserves the order contains all nef adelic arithmetic $\mathbb R$-Cartier divisors. For $\overline D\in C_{\circ\circ}$ and any place in $v\in M_K\cup K(\mathbb C)$, the map $\nabla_\mu(.,\overline D)$ defines a positive functional $\Psi_{\overline D,v}^{\widehat{\deg}}$ on $C^0(X_v^{\mathrm{an}})_+$  which sends $f_v\in C^0(X_v^{\mathrm{an}})$ to $\widehat{\deg}(\overline D{}^d\cdot\overline O(f_v))$. It is an additive functional on $C^0(X_v^{\mathrm{an}})_+$, which coincides with the functional $(\overline D{}^d)_v$ defined in Conventions and terminology \ref{CV:measure} when $\overline D$ is nef.   Therefore, from Theorem \ref{Thm:consequenceDirichlet}, we obtain the following corollary.

\begin{Corollary}\label{Cor:suppD}
Let $\overline D$ be an adelic arithmetic $\mathbb R$-Cartier divisor on $X$. Assume that
$\overline D$ is nef and $\widehat{\deg}(\overline D{}^{d+1})=0$.
If $s$ is an element of $\Rat(X)^{\times}_{\RR}$ with $\overline D{}':=\overline{D} + \widehat{(s)} \geq 0$,
then, for any $v \in M_K \cup K(\CC)$,
\[
\Supp(\Psi_{\overline D,v}^{\widehat{\deg}}) \cap \{x\in X_v^{\mathrm{an}}\,|\,|s|_{g_v}<1\} = \emptyset.
\]

\end{Corollary}

\section{Comparison of the functionals}\label{Subsec:comparison}

Let $\pi:X\rightarrow\Spec K$ be a projective and geometrically integral scheme defined over a number field $K$ and $\overline D$ be an adelic arithmetic $\mathbb R$-Cartier divisor on $X$ such that $D$ is big. In view of the applications of Theorem \ref{Thm:consequenceDirichlet} to different functionals, notably Corollaries \ref{Cor:muasy}, \ref{Cor:volder} and \ref{Cor:suppD}, a natural question is the comparison between the functionals $\Psi_{\overline D,v}^{\widehat{\mu}_{\max}^{\mathrm{asy}}}$, $\Psi_{\overline D,v}^{\widehat{\vol}}$ and $\Psi_{\overline D,v}^{\widehat{\deg}}$. For this purpose, we relate the function $\widehat{\mu}_{\max}^{\mathrm{asy}}(\overline D)$ to the graded linear series of $\overline D$ and show that it is always bounded from below by $\widehat{\vol}(\overline D)/(d+1)\vol(D)$.

\subsection{Asymptotic maximal slope and graded linear series}
\setcounter{Theorem}{0}
If $\overline D=(D,g)$ is an adelic arithmetic $\mathbb R$-Cartier divisor on $X$, we denote by $V(\overline D)$ the $K$-vector subspace of $\Rat(X)$ generated by 
$\hat{H}^0(X,\overline D)$ (see Conventions and terminology \ref{CT:effectivesection}). Let $\zeta$ be an adelic arithmetic $\mathbb R$-Cartier divisor on $\Spec K$ such that $\widehat{\deg}(\zeta)=1$. For any integer $n\geqslant 1$ and any real number $t$, we denote by $V_n^{\zeta,t}(\overline D)$ the $K$-vector subspace of $R(X)$ defined as
\[V_n^{\zeta,t}(\overline D):=V(n\overline D-\pi^*(nt\zeta)).\] 
The direct sum $V_\sbullet^{\zeta,t}(\overline D)$ forms a graded sub-$K$-algebra of $V_\sbullet(D):=\bigoplus_{n\geq 0}H^0(X,nD)$. Moreover, $(V_\sbullet^{\zeta,t})_{t\in\mathbb R}$ forms a multiplicatively concave family of graded linear series of the $\mathbb R$-divisor $D$. Namely, for 
$(t_1,t_2)\in\mathbb R^2$
and $(n,m)\in\mathbb N^2$ one has
\begin{equation}\label{Equ:sous-multp}V_n^{\zeta,t_1}(\overline D)V_m^{\zeta,t_2}(\overline D)\subset V_{n+m}^{\zeta,t}(\overline D)\end{equation}  
where $t=(nt_1+mt_2)/(n+m)$.

For any integer $n\geq 1$, let $\lambda_{n}^{\zeta}(\overline D)$ be the supremum of the set \[\{t\in\mathbb R\,|\, V_n^{\zeta,t}(\overline D)\neq\{0\}\}.\] 
The function $\lambda_n^\zeta(.)$ preserves the order. Namely, if $\overline D$ and $\overline D{}'$ are two adelic arithmetic $\mathbb R$-Cartier divisors on $X$ such that $\overline D\leq\overline D{}'$, then one has $\lambda_n^{\zeta}(\overline D)\leq\lambda_n^{\zeta}(\overline D{}')$. Therefore by \cite[Lemma 2.6]{Boucksom_Chen} (the Hermitian line bundle case), the sequence $(\lambda_n^{\zeta}(\overline D))_{n\geq 1}$ is bounded from above. Moreover,
the relation \eqref{Equ:sous-multp} shows that the sequence $(n\lambda_n^{\zeta}(\overline D))_{n\geq 1}$ is super-additive. Therefore the sequence $(\lambda_n^{\zeta}(\overline D))_{n\geq 1}$ converges in $\mathbb R$. The following proposition relate the limit of the sequence $(\lambda_n^{\zeta}(\overline D))_{n\geq 1}$ to the asymptotic maximal slope of $\overline D$. In particular, the function $\widehat{\mu}_{\max}^{\mathrm{asy}}(.)$ coincides with the one defined in \cite[\S4.2]{Chen10b} in the adelic line bundle setting.

\begin{Proposition}\label{Pro:equivalenceform}
Let $\overline D$ be an adelic arithmetic $\mathbb R$-Cartier divisor such that $D$ is big, then for any $\zeta\in\widehat{\Div}_{\mathbb R}(\Spec K)$ with $\widehat{\deg}(\zeta)=1$ one has
\[\widehat{\mu}_{\max}^{\mathrm{asy}}(\overline{D}) =\lim_{n\rightarrow+\infty}\lambda_n^{\zeta}(\overline D).\]
\end{Proposition}

\begin{proof} We say that an adelic arithmetic $\mathbb R$-Cartier divisor $\overline E$ satisfies to the $\mathbb Q$-Dirichlet property if there is $\varphi\in\Rat(X)_{\mathbb Q}^{\times}$ such that $\overline E+\widehat{(\varphi)}\geq 0$. This condition is stronger than the usual Dirichlet property. We define $\widehat{\mu}_{\max, \QQ}^{\mathrm{asy}}(\overline{D})$ as \[\sup \{ t \in \RR \mid \text{$\overline{D} - t \pi^*(\zeta)$ has the $\mathbb Q$-Dirichlet property}\}.\]

(1) $\displaystyle\lim_{n\rightarrow+\infty}\lambda_n^{\zeta}(\overline D) \leq \widehat{\mu}_{\max, \QQ}^{\mathrm{asy}}(\overline{D})$:
If $n \overline{D} - t n \pi^*(\zeta) + \widehat{(\phi)} \geq 0$ for some $\phi \in \Rat(X)^{\times}$,
then $\overline{D} - t \pi^*(\zeta)$ has the $\QQ$-Dirichlet property, so that
$\lambda^{\zeta}_{n}(\overline{D}) \leq \widehat{\mu}_{\max, \QQ}^{\mathrm{asy}}(\overline{D})$, 
and hence the inequality follows.

\medskip
(2) $\displaystyle\lim_{n\rightarrow+\infty}\lambda_n^{\zeta}(\overline D) \geq \widehat{\mu}_{\max, \QQ}^{\mathrm{asy}}(\overline{D})$:
Let $\epsilon$ be a positive number.
Then there is $t \in \RR$ such that 
\[
\widehat{\mu}_{\max, \QQ}^{\mathrm{asy}}(\overline{D}) - \epsilon \leq t \leq
\widehat{\mu}_{\max, \QQ}^{\mathrm{asy}}(\overline{D})
\]
and $\overline{D} - t \pi^*(\zeta)$ has the $\QQ$-Dirichlet property.
Therefore, $n \overline{D} - t n \pi^*(\zeta) + (\phi) \geq 0$ for some $n \in \ZZ_{>0}$ and
$\phi \in \Rat(X)^{\times}$. Thus,
\[
\widehat{\mu}_{\max, \QQ}^{\mathrm{asy}}(\overline{D}) - \epsilon \leq t \leq
\lambda^{\zeta}_{n}(\overline{D}) \leq \widehat{\mu}_{\max}^{\mathrm{asy}}(\overline{D}),
\]
as required.

\medskip
(3) $\widehat{\mu}_{\max, \QQ}^{\mathrm{asy}}(\overline{D})
\leq \widehat{\mu}_{\max}^{\mathrm{asy}}(\overline{D})$: This is obvious.

\medskip
(4) $\widehat{\mu}_{\max, \QQ}^{\mathrm{asy}}(\overline{D})
\geq \widehat{\mu}_{\max}^{\mathrm{asy}}(\overline{D})$:
Let $\epsilon$ be a positive number. Let $\zeta_{\epsilon}$ be the adelic arithmetic $\RR$-Cartier divisor
on $\Spec K$ given by $\zeta_{\epsilon} := \sum_{\sigma \in K(\CC)} \epsilon [\sigma]$.
We assume that $\overline{D} - t \pi^*(\zeta)$ has the Dirichlet property.
In particular, $\overline{D} - t \pi^*(\zeta)$ is pseudo-effective, so that
$\overline{D} + \pi^*(\zeta_{\epsilon}) - t \pi^*(\zeta)$ 
is big by \cite[Proposition~4.4.2(3)]{MoAdel},
so that $\overline{D} + \pi^*(\zeta_{\epsilon}) - t \pi^*(\zeta)$ has the $\mathbb Q$-Dirichlet property.
Therefore, we have
\[
\widehat{\mu}_{\max}^{\mathrm{asy}}(\overline{D}) \leq \widehat{\mu}_{\max, \QQ}^{\mathrm{asy}}(\overline{D}
+ \pi^*(\zeta_{\epsilon})) = \widehat{\mu}_{\max, \QQ}^{\mathrm{asy}}(\overline{D})
+ \epsilon [K : \QQ]/2,
\]
as desired.
\end{proof}

The following proposition compares the arithmetic maximal slope to a normalized form of the arithmetic volume function.
\begin{Proposition}\label{Pro:comapraisondepsimu}
Let $\overline D$ be an adelic arithmetic $\mathbb R$-Cartier divisor such that $D$ is big and $\widehat{\mu}_{\max}^{\mathrm{asy}}(\overline D)\geq 0$. Then one has
\begin{equation}\label{Equ:inegalitecomparaison}\widehat{\mu}_{\max}^{\mathrm{asy}}(\overline D)\geq\frac{\widehat{\vol}(\overline D)}{(d+1)\vol(D)}.\end{equation}
In particular, if $\widehat{\mu}_{\max}^{\mathrm{asy}}(\overline D)=\widehat{\vol}(\overline D)=0$, then for any 
$\overline E\in\widehat{\Div}_{\mathbb R}(X)_{+}$
one has
\begin{equation}\label{Equ:comparaisondederive}
\nabla_{\overline E}^+\widehat{\mu}_{\max}^{\mathrm{asy}}(\overline D)\geq \frac{1}{(d+1)\vol(D)}\nabla_{\overline E}^+\widehat{\vol}(\overline D),
\end{equation} 
so that,
for any $v\in M_K\cup K(\mathbb C)$, one has
\begin{equation}\label{Equ:inefonctionnel}\Psi_{\overline D,v}^{\widehat{\mu}_{\max}^{\mathrm{asy}}}(f_v)\geq\frac{1}{\vol(D)}\Psi_{\overline D,v}^{\widehat{\vol}}(f_v)\end{equation}
for any non-negative continuous function $f_v$ on $X_v^{\mathrm{an}}$ and hence 
\[\Supp(\Psi_{\overline D,v}^{\widehat{\mu}_{\max}^{\mathrm{asy}}})\supseteq\Supp(\Psi_{\overline D,v}^{\widehat{\vol}}).\]
\end{Proposition}
\begin{proof}
Let $\zeta$ be an adelic arithmetic $\mathbb R$-Cartier divisor on $\Spec K$ such that $\widehat{\deg}(\zeta)=1$. By \cite[Corollary 1.13]{Boucksom_Chen}, one has
\[\widehat{\vol}(\overline D)=(d+1)\int_0^{+\infty}\vol(V^{\zeta,t}_\sbullet)\,\mathrm{d}t.\] 
Therefore
\[\frac{\widehat{\vol}(\overline D)}{(d+1)\vol(D)}=\int_0^{+\infty}\frac{\vol(V_\sbullet^{\zeta,t})}{\vol(V_\sbullet(D))}\,\mathrm{d}t.\]
Moreover, by Proposition \ref{Pro:equivalenceform} we obtain that  $V_n^{\zeta,t}=\{0\}$ once $n\geq 1$ and $t>\widehat{\mu}_{\max}^{\pi}(\overline D)$. Therefore \eqref{Equ:inegalitecomparaison} is proved. 
\end{proof}

\subsection{Poicar\'{e}-Lelong formula and integration of Green function}\label{Sec:intersectionnumber}
\setcounter{Theorem}{0}

Let $X$ be a geometrically integral projective scheme of dimension $d$ defined over a number field $K$, and $\overline D_1,\ldots,\overline D_d$ be integrable adelic arithmetic $\mathbb R$-Cartier divisors on $X$, and $\overline D=(D,g)$ be an arbitrary adelic arithmetic $\mathbb R$-Cartier divisor of $C^0$-type on $X$. The purpose of this section is to establish the following result.
\begin{Proposition}\label{Pro:PLformula}
For each place $v\in M_K\cup K(\mathbb C)$, the Green function $g_v$ is integrable with respect to the signed measure $\overline D_1\cdots\overline D_d$. Moreover, if $[D]$ denotes the $\mathbb R$-coefficient algebraic cycle of dimension $d-1$ corresponding to $D$, then the following relation holds (see Convention and terminology \ref{CV:measure})
\begin{equation}\label{Equ:hauteur}
h(\overline D_1,\cdots,\overline D_d,\overline D;X)=h(\overline D_1,\ldots,\overline D_d;[D])+\sum_{v\in M_K\cup K(\mathbb C)}(\overline D_1\cdots\overline D_d)_v(g_v).
\end{equation}
\end{Proposition}
Note that in the particular case where $D_1,\ldots,D_d,D$ are Cartier divisors, this result has been obtained in \cite[\S5]{Maillot00} and \cite{CL_Thuillier} respectively for hermitian and adelic cases.

Before proving the above proposition, we present several observations as follows. Let $\overline D_1,\ldots,\overline D_d$ 
be integrable adelic arithmetic $\RR$-Cartier divisors on $X$, and $\mathfrak p$ be a place in $M_K$.
In \cite{MoAdel}, the number $\adeg_{\mathfrak p}(\overline{D}_1 \cdots \overline{D}_d ; \phi)$ was defined for an integrable continuous function
$\phi$ on $X^{\an}_{\mathfrak p}$. Moreover,
the set of integrable continuous functions is dense in $C^0(X^{\an}_{\mathfrak p})$
with respect to the supremum norm.
Therefore, one has a natural extension of the functional
$\log\#(O_K/\mathfrak p)\adeg_{\mathfrak p}(\overline{D}_1 \cdots \overline{D}_d; \phi)$ for any continuous function $\phi$, which defines a signed Borel measure on $X_{\mathfrak p}^{\mathrm{an}}$ which we denote by $(\overline D_1\cdots\overline D_d)_{\mathfrak p}$. Similarly, for any $\sigma\in K(\mathbb C)$, the product of currents $\frac 12c_1(D_1,g_{1,\sigma})\wedge\cdots\wedge c_1(D_d,g_{d,\sigma})$ defines a signed Borel measure on $X_\sigma^{\mathrm{an}}$ which we denote by $(\overline D_1\cdots\overline D_d)_{\sigma}$ . See Convention and terminology \ref{CV:measure} for a presentation of the construction of these measures in the language of arithmetic intersection product. In particular, the proposition is true in the special case where $D=0$.  

\begin{proof}[Proof of Proposition \ref{Pro:PLformula}]
Note that both side the equality \eqref{Equ:hauteur} are multilinear with respect to the vector $(\overline D_1,\cdots,\overline D_d,\overline D)$. Therefore we may assume that $\overline D_1,\ldots,\overline D_d$ are relatively nef and $D$ is ample without loss of generality. Moreover, for any place $v\in M_K\cup K(\mathbb C)$ two $D$-Green functions on $X_v^{\mathrm{an}}$ differ by a continuous function on $X_v^{\mathrm{an}}$. Therefore it suffice to prove the proposition for an arbitrary choice of adelic $D$-Green functions, and the general case follows by the linearity of the problem and the particular case where $D=0$. In particular, we can assume without loss of generality that $\overline D$ is a \emph{relatively ample arithmetic Cartier divisor}, namely $D$ comes from an ample line bundle $L$ on $X$ and the adelic structure on $D$ comes from an ample integral model of $(X,L)$ equipped with semi-positive metrics at infinite places.

We shall prove the following claim by induction on $k$. Note that the case where $k=d+1$ is just the result of the proposition itself.
\begin{Claim}Assume that $\overline D_i$ is relatively nef for any $i\in\{1,\ldots,d\}$ .
Let $k$ be an index in $\{1,\ldots,d+1\}$. Then the assertion of the proposition holds provided that each $\overline D_i$ ($k\leq i\leq d$) can be written as a positive linear combination of ample Cartier divisors equipped with Green functions of $C^0\cap\mathrm{PSH}$-type.
\end{Claim}
The claim in the case where $k=1$ is classical, which results from \cite[Theorem 4.1]{CL_Thuillier} by multilinearity. In the following, we verify that the claim for $k$ implies the same claim for $k+1$. We choose an $\mathbb R$-Cartier divisor $E_k$ such that $E_k$ can be written as a positive linear combination of ample Cartier divisors and $D_k'=E_k+D_k$ is an ample Cartier divisor. We choose suitable $E_k$-Green functions such that $\overline E_k$ can be written as a positive linear combination of ample Cartier divisors equipped with Green functions of $C^0\cap\mathrm{PSH}$-type. Then $\overline D_k'=\overline E_k+\overline D_k$ is an ample Cartier divisor equipped with Green functions of $C^0\cap\mathrm{PSH}$-type. The induction hypothesis then implies that the claim holds for $\overline D_1,\ldots,\overline D_{k-1},\overline E_k,\overline D_{k+1},\ldots,\overline D_d$ and for $\overline D_1,\ldots,\overline D_{k-1},\overline D_k',\overline D_{k+1},\ldots,\overline D_d$. We then conclude by the multilinearity of the problem.
\end{proof}

\subsection{Intersection measure and comparison with $\Psi_{\overline D,v}^{\widehat{\mu}_{\max}^{\mathrm{asy}}}$}
\setcounter{Theorem}{0}

Similarly to the results obtained in the previous subsection, in the case where $\overline D$ is nef and $D$ is big, the linear functional $((d+1)\vol(D))^{-1}\Psi_{\overline D,v}^{\widehat{\deg}}$ is bounded from above by $\Psi_{\overline D,v}^{\widehat{\mu}_{\max}^{\mathrm{asy}}}$ provided that $\widehat{\deg}(\overline D{}^{d+1})=\widehat{\mu}_{\max}^{\mathrm{asy}}(\overline D)=0$ (we can actually prove that they are equal). This comparison uses a generalization of the positive intersection product to the framework of adelic arithmetic $\mathbb R$-Cartier divisors.  



Let $\overline D$ be a big adelic arithmetic $\mathbb R$-Cartier divisor on $X$. We denote by $\Theta(\overline D)$ the set of couples $(\nu,\overline N)$, where $\nu:X'\rightarrow X$ is a birational projective morphism and $\overline N$ is a nef adelic arithmetic $\mathbb R$-Cartier divisor on $X'$ such that \[\hat{H}^0(X',t(\nu^*(\overline D)-\overline N))\neq\{0\}\] for some $t>0$. We then define a functional $\langle\overline D{}^d\rangle$ on the cone $\widehat{\mathrm{Nef}}_{\mathbb R}(X)$ of nef adelic arithmetic $\mathbb R$-Cartier divisors on $X$ as
\[\forall\,\overline A\in\widehat{\mathrm{Nef}}_{\mathbb R}(X),\quad \langle\overline D{}^d\rangle\cdot\overline A:=\sup_{(\nu,\overline N)\in\Theta}\widehat{\deg}(\overline N{}^{d}\cdot\nu^*(\overline A)).
 \]
If the set $\Theta(\overline D)$ is empty, then the value of $\langle\overline D{}^d\rangle\cdot\overline A$ is defined to be zero by convention. The set $\Theta(\overline D)$ is preordered in the following way~: \[(\nu_1:X_1\rightarrow X,\overline N_1)\geq(\nu_2:X_2\rightarrow X,\overline N_2)\] if and only if there exists a birational modification $\nu':X'\rightarrow X$ over both $X_1$ and $X_2$ such that $\hat{H}^0(X',t(p_1^*\overline N_1-p_2^*\overline N_2))\neq\{0\}$ for some $t>0$, where $p_i:X'\rightarrow X_i$ ($i=1,2$) are structural morphisms (which are birational projective morphisms such that $\nu_1p_1=\nu_2p_2=\nu'$). By the same method of \cite[Proposition 3.3]{Chen11}, one can prove that $\Theta(\overline D)$ is filtered with respect to this preorder and hence $\langle\overline D{}^d\rangle$ is an additive and positively homogeneous functional on $\widehat{\mathrm{Nef}}_{\mathbb R}(X)$ and hence extends to a linear form on the vector space $\widehat{\mathrm{Int}}_{\mathbb R}(X)$ of integrable adelic arithmetic $\mathbb R$-Cartier divisors. Finally, by \cite[Theorem 3.3.7]{MoAdel}, one can extends by continuity the functional $\langle\overline D{}^d\rangle$ to the whole space $\widehat{\mathrm{Div}}_{\mathbb R}(X)$ of adelic arithmetic $\mathbb R$-Cartier divisors such that $\langle\overline D{}^d\rangle\cdot\overline E\geq 0$ if $\overline E\geq 0$.

The comparison between $\vol(D)^{-1}\Psi_{\overline D,v}^{\widehat{\deg}}$ and $\Psi_{\overline D,v}^{\widehat{\mu}_{\max}^{\mathrm{asy}}}$ comes from the following variant of \cite[Theorem 2.2]{Yuan07} in the adelic arithmetic $\mathbb R$-Cartier divisor setting. 

\begin{Theorem}
Let $\overline{D} = (D,g)$ and $\overline{L} = (L, h)$ be
nef adelic arithmetic $\RR$-Cartier divisors. Then
\begin{equation}\label{Equ:sius ineq}
\avol(\overline{D} - \overline{L}) \geq
\adeg(\overline{D}{}^{d+1}) - (d+1) \adeg(\overline{D}{}^d \cdot \overline{L}).
\end{equation}
\end{Theorem}

\begin{proof}
By \cite[Definition~2.1.6 and Proposition~2.1.7]{MoAdel},
there are sequences 
\[
\{ (\XXX_n, \DDD_n) \}_{n=1}^{\infty}
\quad\text{and}\quad
\{ (\XXX_n, \LLL_n) \}_{n=1}^{\infty}
\]
of models $(X,D)$ and $(X,L)$,
respectively, with the following properties:
\begin{enumerate}
\item
$\overline{\DDD}_n = \left(\DDD_n, \sum_{\sigma \in K(\CC)} g_{\sigma}\right)$ and
$\overline{\LLL}_n = \left(\LLL_n, \sum_{\sigma \in K(\CC)} h_{\sigma}\right)$ are
nef for each $n$.

\item
If we set 
\[
\left(0, \sum_{\mathfrak p \in M_K} g'_{n,\mathfrak p}[\mathfrak p] \right)= \overline{D}-
\overline{\DDD}^{\ad}_n
\quad\text{and}\quad
\left( 0, \sum_{\mathfrak p \in M_K} h'_{n,\mathfrak p}[\mathfrak p] \right)
= \overline{L} - \overline{\LLL}^{\ad}_n,
\]
then
$\lim_{n\to\infty} \Vert g'_{n,\mathfrak p} \Vert_{\sup} = 0$
and
$\lim_{n\to\infty} \Vert h'_{n,\mathfrak p} \Vert_{\sup} = 0$.
\end{enumerate}
Therefore, by \cite[Theorem~5.2.1]{MoAdel},
it is sufficient to see 
the case where 
$\overline{D} = \overline{\DDD}$ and $\overline{L} = \overline{\LLL}$
for nef arithmetic $\RR$-Cartier divisors $\overline{\DDD}$ and
$\overline{\LLL}$ on some arithmetic variety $\XXX$.

Let $\overline{\AAA}$ be an ample arithmetic $\RR$-Cartier 
divisor on $\XXX$.
If Siu's inequality holds for $\overline{\DDD} + \epsilon \overline{\AAA}$
and $\overline{\LLL} + \epsilon \overline{\AAA}$ ($\epsilon  > 0$), then, by using the continuity of
the volume function, we have the assertion for our case, 
so that
we may assume that $\DDD$ and $\LLL$ are ample.
Thus we can set
\[
\DDD = a_1 \DDD'_1 + \cdots + a_l \DDD'_l
\quad\text{and}\quad
\LLL = b_1 \LLL'_1 + \cdots + b_r \LLL'_r,
\]
where $\LLL'_1,\ldots,\LLL'_l,\DDD'_1,\ldots,\DDD'_r$ are ample Cartier divisors on $\XXX$ and
\[
a_1, \ldots, a_l, b_1, \ldots, b_r \in \RR_{>0}.
\]
Let $g'_i$ (resp. $h'_j$) be a $\DDD'_i$-Green function 
(resp. $\LLL'_j$-Green function)
such that $\overline{\DDD}'_i = (\DDD'_i, g'_i)$ (resp. $\overline{\LLL}'_j = (\LLL'_j, h'_j)$)
is ample.
We set
\[
\overline{\DDD} = a_1 \overline{\DDD}'_1 + \cdots + a_l \overline{\DDD'}_l + (0, \phi)
\quad\text{and}\quad 
\overline{\LLL} = b_1 \overline{\LLL}'_1 + \cdots + b_r \overline{\LLL}'_r + (0, \psi).
\]
Moreover, for $a'_1, \ldots, a'_l, b'_1, \ldots, b'_r \in \RR$,
we set
\[
\overline{\DDD}_{a'_1,\ldots,a'_l} = a'_1 \overline{\DDD}'_1 + \cdots + a'_l \overline{\DDD}'_l + (0, \phi)
\ \text{and}\ 
\overline{\LLL}_{b'_1,\ldots,b'_r} = b'_1 \overline{\LLL}'_1 + \cdots + b'_r \overline{\LLL}'_r + (0, \psi).
\]
Note that if $a'_1 \geq a_1, \ldots, a'_i \geq a_l$ and
$b'_1 \geq b_1,\ldots,b'_r \geq b_r$, then
$\overline{\DDD}_{a'_1,\ldots,a'_l}$ and
$\overline{\LLL}_{b'_1, \ldots,b'_r}$ are nef, so that,
using the continuity of the volume function together with
Siu's inequality (cf. \cite[Theorem 2.2]{Yuan07}) for nef arithmetic $\QQ$-divisors,
we have the assertion.
\end{proof}


\begin{Corollary}
Let $\overline D$ be a relatively nef adelic arithmetic $\mathbb R$-Cartier divisor on $X$ such that $D$ is big. If \begin{equation}\label{Equ:egalitedehaut}\widehat{\mu}_{\max}^{\mathrm{asy}}(\overline D)=\frac{\widehat{\deg}(\overline D{}^{d+1})}{(d+1)\vol(D)},\end{equation} then one has
\begin{equation}\label{Equ:derivedirec}
\forall\,\overline E\in\widehat{\Div}_{\mathbb R}(X),\quad\nabla_{\overline E}^+\widehat{\mu}_{\max}^{\mathrm{asy}}(\overline D)=\frac{\widehat{\deg}(\overline D{}^d\cdot\overline E)}{\vol(D)}.
\end{equation}
In particular,
\begin{equation}
\label{Equ:egalitedemesure}
\Psi_{\overline D,v}^{\widehat{\mu}_{\max}^{\mathrm{asy}}}=\frac{(\overline D{}^d)_v}{\vol(D)}\end{equation}
for any $v\in M_K\cup K(\mathbb C)$.
\end{Corollary}
\begin{proof} Since the map $\overline E\mapsto\nabla_{\overline E}^+{\widehat{\mu}_{\max}^{\mathrm{asy}}}(\overline D)$ from $\widehat{\Div}_{\mathbb R}(X)$ to $\mathbb R\cup\{+\infty\}$ is super-additive and $\overline E\mapsto\widehat{\deg}(\overline D{}^d\cdot\overline E)$ is a linear functional, it suffice to establish the inequality (see 
\cite[Remark~4.3]{Chen11})
\[\forall\,\overline E\in\widehat{\Div}_{\mathbb R}(X),\quad \nabla_{\overline E}^+{\widehat{\mu}_{\max}^{\mathrm{asy}}}(\overline D)\geq \frac{\widehat{\deg}(\overline D{}^d\cdot\overline E)}{\vol( D)}.\]
Note that both the condition \eqref{Equ:egalitedehaut} and the assertion \eqref{Equ:egalitedemesure} remain equivalent if one replaces $\overline D$ by $\overline D+\pi^*(\zeta)$ with $\zeta\in\widehat{\Div}_{\mathbb R}(\Spec K)$. Therefore we may assume that $\overline D$ is nef and big without loss of generality. In this case one has 
\[\forall\,\overline E\in\widehat{\Div}_{\mathbb R}(X),\quad \widehat{\deg}(\overline D{}^d\cdot\overline E)=\langle\overline D{}^d\rangle\cdot\overline E\]
since $\overline D$ is nef.
We shall
actually establish the equality \[\nabla_{\overline E}^+\widehat{\vol}(\overline D)=(d+1)\langle\overline D{}^d\rangle\cdot\overline E\] for any $\overline E\in\widehat{\mathrm{Int}}_{\mathbb R}(X)$. We choose $\overline M\in\widehat{\mathrm{Nef}}_{\mathbb R}(X)$ such that $\overline M-\overline E$ and $\overline M+\overline E$ are nef and big and $\overline M-\overline D$ is big. Then for any $(\nu,\overline N)\in\Theta(\overline D)$, by \eqref{Equ:sius ineq} one has
\begin{equation}
\begin{aligned}
\widehat{\vol}(\overline D+t\overline E) & \geq\widehat{\vol}(\overline N+t\nu^*(\overline E))=\widehat{\vol}((\overline N+t\nu^*(\overline M))-t\nu^*(\overline M-\overline E))\\
&\geq\widehat{\deg}((\overline N+t\nu^*(\overline M)^{d+1})\\
& \qquad\qquad -(d+1)t\widehat{\deg}((\overline N+t\nu^*(\overline M))^d\cdot\nu^*(\overline M-\overline E))\\&=\widehat{\deg}(\overline N{}^{d+1})+t(d+1)\widehat{\deg}(\overline N{}^d\cdot\overline E)+O(t^2),
\end{aligned}
\end{equation}
where the implicit constant in $O(t^2)$ only depends on $\widehat{\vol}(\overline M)=\widehat{\deg}(\nu^*(\overline M)^{d+1})$.
We then deduce
\[\widehat{\vol}(\overline D+t\overline E)\geq\widehat{\vol}(\overline D)+t(d+1)\langle\overline D{}^d\rangle\cdot\overline E+O(t^2), \]
which implies that  \[\nabla^+_{\overline E}\widehat{\vol}(\overline D)\geq(d+1)\langle\overline D{}^d\rangle\cdot\overline E\] By the continuity of the linear functional $\overline E\mapsto\langle\overline D{}^d\rangle\cdot\overline E$, we obtain that this inequality holds for general $\overline E\in\widehat{\Div}_{\mathbb R}(X)$. Finally, by the log-concavity of the arithmetic volume function, the  functional $\overline E\mapsto\nabla_{\overline E}^+\widehat{\vol}(\overline D)$ is super-additive (see 
\cite[Remark~4.2]{Chen11}). Therefore one obtains $\nabla_{\overline E}^+\widehat{\vol}(\overline D)=(d+1)\langle\overline D{}^d\rangle\cdot\overline E$. The result is thus proved by using the relation \eqref{Equ:comparaisondederive}.
\end{proof}

\subsection{Comparison with the distribution of non-positive points}
\setcounter{Theorem}{0}
In this subsection, we compare the functional approach and the distribution of non-positive points in the particular case where the set of non-positive points is Zariski dense. The main point is an equidistribution argument. Let $X$ be a projective and geometrically integral variety over a number field $K$. Let $\overline D$ be an adelic arithmetic $\mathbb R$-Cartier divisor on $X$. We assume that $\overline D$ is nef and $D$ is big. 

\begin{Proposition}
Assume that the set $X(\overline K)_{\leq 0}^{\overline D}$ is Zariski dense, then for any place $v\in M_K\cup K(\mathbb C)$ one has
\[\Psi_{\overline D,v}^{\widehat{\mu}_{\max}^{\mathrm{asy}}}=\frac{(\overline D{}^d)_v}{\vol(D)}\]
and
\[\Supp_{\mathrm{ess}}(X(\overline K)_{\leq 0}^{\overline D})_v^{\mathrm{an}}\supseteq\Supp\Psi_{\overline D,v}^{\widehat{\mu}_{\max}^{\mathrm{asy}}}.\]
\end{Proposition}
\begin{proof}
Since the set $X(\overline K)_{\leq 0}^{\overline D}$ is Zariski dense, we obtain that the essential minimum
of the height function $h_{\overline D}(.)$ is non-positive. However, since $\overline D$ is nef one has $\widehat{\deg}(\overline D{}^{d+1})=\widehat{\vol}(\overline D)$ and therefore
\[0=\widehat{\mu}_{\mathrm{ess}}(\overline D)\geq\widehat{\mu}_{\max}^{\mathrm{asy}}(\overline D)\geq\frac{\widehat{\vol}(\overline D)}{(d+1)\vol(D)}\geq 0,\]
which implies that  $\widehat{\deg}(\overline D{}^{d+1})=\widehat{\vol}(\overline D)=0$.

Let $S=(x_n)_{n\geq 1}$ be a generic sequence in $X(\overline K)_{\leq 0}^{\overline D}$. For any adelic arithmetic $\mathbb R$-Cartier divisor $\overline E$ on $X$ we define
\[\Phi_S(\overline E):=\liminf_{n\rightarrow+\infty}h_{\overline E}(x_n).\]
This function takes value in $\mathbb R\cup\{+\infty\}$ on the cone $\Theta$ of adelic arithmetic $\mathbb R$-Cartier divisors $\overline E$ such that $E$ is big. The function $\Phi_S(.):\Theta\rightarrow\mathbb R\cup\{+\infty\}$ is also super-additive. Moreover, one has $\Phi_S(\overline E)\geq\widehat{\mu}_{\max}^{\mathrm{asy}}(\overline E)$ for any $\overline E\in\Theta$ and $\Phi_S(\overline D)=\widehat{\mu}_{\max}^{\mathrm{asy}}(\overline D)=0$. Therefore one has
\[\forall\,\overline E\in\widehat{\Div}_{\mathbb R}(X),\quad\nabla_{\overline E}^+\Phi_S(\overline D)\geq\nabla_{\overline E}^+\widehat{\mu}_{\max}^{\mathrm{asy}}(\overline D).\]
By \eqref{Equ:derivedirec} and \cite[Proposition 4.3]{Chen11}, one obtains
\[\forall\,\overline E\in\widehat{\Div}_{\mathbb R}(X),\quad\nabla_{\overline E}^+\Phi_S(\overline D)=\nabla_{\overline E}^+\widehat{\mu}_{\max}^{\mathrm{asy}}(\overline D)=\frac{\widehat{\deg}(\overline D{}^d\cdot\overline E)}{\vol(D)}.\]
This relation implies that, for any $\overline E\in\widehat{\Div}_{\mathbb R}(X)$, the sequence $(h_{\overline E}(x_n))_{n\geq 1}$ actually converges to $\vol(D)^{-1}\widehat{\deg}(\overline D{}^d\cdot\overline E)$. In fact, one has $h_{\overline D}(x_n)=0$ for any $n\in\mathbb N$, $n\geq 1$. Therefore $\nabla_{\overline E}^+\Phi_S(\overline D)=\Phi_S(\overline E)$. In particular, one has $\Phi_S(-\overline E)=-\Phi_S(\overline E)$, which implies the convergence of the sequence $(h_{\overline E}(x_n))_{n\geq 1}$.

Suppose
$\Supp_{\mathrm{ess}}(X(\overline K)_{\leq 0}^{\overline D})_v^{\mathrm{an}}\not\supseteq\Supp\Psi_{\overline D,v}^{\widehat{\mu}_{\max}^{\mathrm{asy}}}$, that is,
there is $w_v \in \Supp\Psi_{\overline D,v}^{\widehat{\mu}_{\max}^{\mathrm{asy}}} \setminus 
\Supp_{\mathrm{ess}}(X(\overline K)_{\leq 0}^{\overline D})_v^{\mathrm{an}}$.
As $w_v \not\in \Supp_{\mathrm{ess}}(X(\overline K)_{\leq 0}^{\overline D})_v^{\mathrm{an}}$,
there is a proper subscheme $Y$ of $X$ such that $w_v \not\in \overline{\Delta(X(\overline K)_{\leq 0}^{\overline D}; Y)_v^{\an}}$.
Let $f_v$ be a non-negative continuous function on $X_v^{\an}$ such that
$f_v(w_v) = 1$ and $f_v \equiv 0$ on $\overline{\Delta(X(\overline K)_{\leq 0}^{\overline D}; Y)_v^{\an}}$.

\begin{Claim}
For $x \in X(\overline K)_{\leq 0}^{\overline D} \setminus Y(\overline{K})$, we have
$h_{\overline{O}(f_v)}(x) = 0$.
\end{Claim}

\begin{proof}
If $v \in M_K$, the assertion is obvious, so that assume $v \in K(\CC)$.
By the definition of $\overline{O}(f_v)$ (cf. Conventions and terminology~\ref{CV:adelic:arith:div}), 
\[
4[K(x) : K] h_{\overline{O}(f_v)}(x) = \sum_{w \in O_{v}(x)} f_v(w) + \sum_{w' \in O_{\bar{v}}(x)} f_v(F_{\infty}(w')).
\]
Note that $F_{\infty}(O_{\bar{v}}(x)) = O_{v}(x)$, and
hence the assertion follows.
\end{proof}

By the previous observation,
\[\lim_{n\rightarrow \infty}h_{\overline O(f_v)}(x_n)=\Phi_S(\overline O(f_v))=\frac{(\overline D{}^d)_v(f_v)}{\vol(D)}>0.\]
On the other hand,
as $S = (x_n)_{n \geq 1}$ is generic, there is a subsequence $S' = (x_{n_i})$ such that
$x_{n_i} \not\in Y(\overline{K})$ for all $i$, so that, by the above claim,
$\lim_{i\rightarrow \infty}h_{\overline O(f_v)}(x_{n_i}) = 0$.
This is a contradiction.
\end{proof}

\renewcommand{\theTheorem}{\arabic{section}.\arabic{Theorem}}
\renewcommand{\theClaim}{\arabic{section}.\arabic{Theorem}.\arabic{Claim}}

\section{Extension of the asymptotic maximal slope}
\label{Sec:extension}
In this section, we extends the function of the asymptotic maximal slope to the whole space $\widehat{\Div}_{\mathbb R}(X)$ of adelic arithmetic $\mathbb R$-Cartier divisors. Let $\overline D$ be an adelic arithmetic $\mathbb R$-Cartier divisor on $X$. We define $\widehat{\mu}_{\max}^{\mathrm{asy}}(\overline D)$ to be
\begin{equation}\label{Equ:mumaxasygene}\inf_{\overline D_0\in\Theta}\lim_{t\rightarrow+\infty}\big(\widehat{\mu}_{\max}^{\mathrm{asy}}(t\overline D_0+\overline D)-t\widehat{\mu}_{\max}^{\mathrm{asy}}(\overline D_0)\big)\in\mathbb R\cup\{-\infty\},\end{equation}
where $\Theta$ denotes the set of all adelic arithmetic 
$\mathbb R$-Cartier divisors $\overline E$ such that $E$ is big. 
Note that if $D$ is big, then the value \eqref{Equ:mumaxasygene} coincides with the maximal asymptotic maximal slope of $\overline D$. In fact, for any $\overline D_0\in\Theta$ one has 
\[\widehat{\mu}_{\max}^{\mathrm{asy}}(t\overline D_0+\overline D)-t\widehat{\mu}_{\max}^{\mathrm{asy}}(\overline D_0)\geq \widehat{\mu}_{\max}^{\mathrm{asy}}(t\overline D_0)+\widehat{\mu}_{\max}^{\mathrm{asy}}(\overline D)-t\widehat{\mu}_{\max}^{\mathrm{asy}}(\overline D_0)=\widehat{\mu}_{\max}^{\mathrm{asy}}(\overline D).\]
Therefore the infimum is attained at $\overline D_0=\overline D$ and coincides with $\widehat{\mu}_{\max}^{\mathrm{asy}}(\overline D)$.

The extended function also verifies the good properties such as positive homogenity, super-additivity etc. We resumes these properties in the following proposition.
\begin{Proposition}\label{Pro:prorietiesdemuasy}\begin{enumerate}[(1)]
\item Let $\overline D$ be an adelic arithmetic $\mathbb R$-Cartier divisor on $X$. For any $\lambda\geq 0$ one has $\widehat{\mu}_{\max}^{\mathrm{asy}}(\lambda\overline D)=\lambda\widehat{\mu}_{\max}^{\mathrm{asy}}(\overline D)$.
\item Let $\overline D_1$ and $\overline D_2$ be two adelic arithmetic $\mathbb R$-Cartier divisors on $X$. One has \[\widehat{\mu}_{\max}^{\mathrm{asy}}(\overline D_1+\overline D_2)\geq \widehat{\mu}_{\max}^{\mathrm{asy}}(\overline D_1)+\widehat{\mu}_{\max}^{\mathrm{asy}}(\overline D_2).\]
\item Let $\overline D_1$ and $\overline D_2$ be two adelic arithmetic $\mathbb R$-Cartier divisors on $X$. If $\overline D_1\geq\overline D_2$, then $\widehat{\mu}_{\max}^{\mathrm{asy}}(\overline D_1)\geq \widehat{\mu}_{\max}^{\mathrm{asy}}(\overline D_2)$.
\item If $\overline D$ is an adelic arithmetic $\mathbb R$-Cartier divisor on $X$ and $\zeta$ is an adelic arithmetic $\mathbb R$-Cartier divisor on $\Spec K$, one has 
\[\widehat{\mu}_{\max}^{\mathrm{asy}}(\overline D+\pi^*(\zeta))=\widehat{\mu}_{\max}^{\mathrm{asy}}(\overline D)+\widehat{\deg}(\zeta).\]
\item For any adelic arithmetic $\mathbb R$-Cartier divisor $\overline D$ on $X$ and any $\varphi\in\Rat(X)^{\times}_{\mathbb R}$ one has 
\[\widehat{\mu}_{\max}^{\mathrm{asy}}(\overline D+\widehat{(\varphi)})=\widehat{\mu}_{\max}^{\mathrm{asy}}(\overline D).\]
\end{enumerate}
\end{Proposition}
\begin{proof}
(1) The equality is trivial when $\lambda=0$. In the following, we assume that $\lambda>0$. For any $\overline{D}_0\in\Theta$ one has
\begin{multline*}
\lim_{t\rightarrow+\infty}\big(\widehat{\mu}_{\max}^{\mathrm{asy}}(t\overline D_0+\lambda\overline D)-t\widehat{\mu}_{\max}^{\mathrm{asy}}(\overline D_0)\big) \\
=\lim_{t\rightarrow+\infty}\big(\widehat{\mu}_{\max}^{\mathrm{asy}}(\lambda t\overline D_0+\lambda\overline D)-\lambda t\widehat{\mu}_{\max}^{\mathrm{asy}}(\overline D_0)\big) \\
=\lambda\lim_{t\rightarrow+\infty} \big(\widehat{\mu}_{\max}^{\mathrm{asy}}(t\overline D_0+\overline D)-t\widehat{\mu}_{\max}^{\mathrm{asy}}(\overline D_0)\big)
\end{multline*}
By taking the infimum with respect to $\overline D_0$, one obtains the result.

(2) Let $\overline D_0$ be an element in $\Theta$. For sufficiently positive $t$, one has
\begin{multline*}
\lim_{t\rightarrow+\infty}\widehat{\mu}_{\max}^{\mathrm{asy}}(2t\overline D_0+\overline D_1+\overline D_2)-2t\widehat{\mu}_{\max}^{\mathrm{asy}}(\overline D_0)\\
\hspace{-10em} \geq \lim_{t\rightarrow+\infty}\widehat{\mu}_{\max}^{\mathrm{asy}}(t\overline D_0+\overline D_1)-t\widehat{\mu}_{\max}^{\mathrm{asy}}(\overline D_0) \\
+\lim_{t\rightarrow+\infty}\widehat{\mu}_{\max}^{\mathrm{asy}}(t\overline D_0+\overline D_1)-t\widehat{\mu}_{\max}^{\mathrm{asy}}(\overline D_0)
\geq\widehat{\mu}_{\max}^{\mathrm{asy}}(\overline D_1)+\widehat{\mu}_{\max}^{\mathrm{asy}}(\overline D_2).
\end{multline*}
Since $\overline D_0$ is arbitrary, one obtains the result.

(3) Let $\overline D_0$ be an element in $\Theta$. For sufficiently positive number $t$ one has
\begin{multline*}
\lim_{t\rightarrow+\infty}\widehat{\mu}_{\max}^{\mathrm{asy}}(t\overline D_0+\overline D_1)-t\widehat{\mu}_{\max}^{\mathrm{asy}}(\overline D_0)\\
\geq \lim_{t\rightarrow+\infty}\widehat{\mu}_{\max}^{\mathrm{asy}}(t\overline D_0+\overline D_2)-t\widehat{\mu}_{\max}^{\mathrm{asy}}(\overline D_0)\geq\widehat{\mu}_{\max}^{\mathrm{asy}}(\overline D_2).
\end{multline*}
Since $\overline D_0$ is arbitrary, one obtains $\widehat{\mu}_{\max}^{\mathrm{asy}}(\overline D_1)\geq\widehat{\mu}_{\max}^{\mathrm{asy}}(\overline D_2)$.

(4) For any $\overline D_0\in\Theta$ and any sufficiently positive number $t$, one has
\[\widehat{\mu}_{\max}^{\mathrm{asy}}(t\overline D_0+\overline D+\pi^*(\zeta))-t\widehat{\mu}_{\max}^{\mathrm{asy}}(\overline D_0)=\widehat{\mu}_{\max}^{\mathrm{asy}}(t\overline D_0+\overline D)-t\widehat{\mu}_{\max}^{\mathrm{asy}}(\overline D_0)+\widehat{\deg}(\zeta).\]
By passing to limit when $t$ tends to the infinity and then by taking the infimum with respect to $\overline D_0$, one obtains $\widehat{\mu}_{\max}^{\mathrm{asy}}(\overline D+\pi^*(\zeta))=\widehat{\mu}_{\max}^{\mathrm{asy}}(\overline D)+\widehat{\deg}(\zeta)$.

(5) Let $\overline D_0$ be an element in $\Theta$. For sufficiently positive number $t$, one has
\[\widehat{\mu}_{\max}^{\mathrm{asy}}(t\overline D_0+\overline D+\widehat{(\varphi)})-t\widehat{\mu}_{\max}^{\mathrm{asy}}(\overline D_0)=\widehat{\mu}_{\max}^{\mathrm{asy}}(t\overline D_0+\overline D)-t\widehat{\mu}_{\max}^{\mathrm{asy}}(\overline D_0).\]
Therefore $\widehat{\mu}_{\max}^{\mathrm{asy}}(\overline D+\widehat{(\varphi)})=\widehat{\mu}_{\max}^{\mathrm{asy}}(\overline D)$.

\end{proof}

The following is a criterion for the pseudo-effectivity of adelic arithmetic $\mathbb R$-Cartier divisors, which is a generalization of Proposition~\ref{Pro:big:cripseudoeff}.

\begin{Proposition}\label{Pro:cripseudoeff}
Let $\overline D$ be an adelic arithmetic $\mathbb R$-Cartier divisor. Then $\overline D$ is pseudo-effective if and only if $D$ is pseudo-effective and $\widehat{\mu}_{\max}^{\mathrm{asy}}(\overline D)\geq 0$.
\end{Proposition}
\begin{proof}
Assume that $\overline D$ is pseudo-effective, then $D$ is a pseudo-effective $\mathbb R$-divisor. Moreover, for any $\overline D_0\in\Theta$, there exists $\zeta\in\widehat{\Div}_{\mathbb R}(\Spec K)$ such that $\overline D_1=\overline D_0+\pi^*(\zeta)$ is big. Therefore, for any $\varepsilon>0$, the adelic arithmetic $\mathbb R$-Cartier divisor $\varepsilon(\overline D_0+\pi^*(\zeta))+\overline D$ is big. Hence for $t>\varepsilon$ one has
\[\begin{split}&\quad\;\widehat{\mu}_{\max}^{\mathrm{asy}}(t\overline D_0+\overline D)-t\widehat{\mu}_{\max}^{\mathrm{asy}}(\overline D_0)=\widehat{\mu}_{\max}^{\mathrm{asy}}(t\overline D_1+\overline D)-t\widehat{\mu}_{\max}^{\mathrm{asy}}(\overline D_1)\\
&\geq (t-\varepsilon)\widehat{\mu}_{\max}^{\mathrm{asy}}(\overline D_1)-t\widehat{\mu}_{\max}^{\mathrm{asy}}(\overline D_1)=-\varepsilon\widehat{\mu}_{\max}^{\mathrm{asy}}(\overline D_1).
\end{split}\]
Since $\varepsilon$ is arbitrary, we obtain that 
\[\lim_{t\rightarrow+\infty}\big(\widehat{\mu}_{\max}^{\mathrm{asy}}(t\overline D_0+\overline D)-t\widehat{\mu}_{\max}^{\mathrm{asy}}(\overline D_0)\big)\geq 0.\]

Conversely assume that $D$ is pseudo-effective and  $\widehat{\mu}_{\max}^{\mathrm{asy}}(\overline D)\geq 0$. If $\overline D{}'$ is a big adelic arithmetic $\mathbb R$-Cartier divisor, then $D+D'$ is big since $D$ is pseudo-effective and $D'$ is big. Moreover, one has \[\widehat{\mu}_{\max}^{\mathrm{asy}}(\overline D+\overline D{}')\geq \widehat{\mu}_{\max}^{\mathrm{asy}}(\overline D)+\widehat{\mu}_{\max}^{\mathrm{asy}}(\overline D{}')>0.\]
Hence $\overline D+\overline D{}'$ is big by Proposition~\ref{Pro:big:cripseudoeff}. 
Therefore $\overline D$ is pseudo-effective. 
\end{proof}

The results which we have obtained in \S\ref{Sub:formalDirichlet} can be applied to the extended function $\widehat{\mu}_{\max}^{\mathrm{asy}}(.)$. Let $C_\circ$ be the cone of all pseudo-effective adelic arithmetic $\mathbb R$-Cartier divisors and $V=\widehat{\Div}_{\mathbb R}(X)$. Then the cone $C_\circ$ satisfies the conditions (a)--(c) of \S\ref{Sub:formalDirichlet}. Moreover, Proposition \ref{Pro:cripseudoeff} shows that the restriction of $\widehat{\mu}_{\max}^{\mathrm{asy}}$ on $C_\circ$ is a real valued function. By Proposition \ref{Pro:prorietiesdemuasy}, this function verifies the conditions (1)--(3) of \S\ref{Sub:formalDirichlet}.
Thus we obtain the following corollary of Theorem \ref{Thm:consequenceDirichlet}.

\begin{Corollary}\label{Cor:muasy2}
Let $\overline D$ be an adelic arithmetic $\mathbb R$-Cartier divisor such that $D$ is pseudo-effective and that $\widehat{\mu}_{\max}^{\mathrm{asy}}(\overline D)=0$. If $s$ is an element of $\Rat(X)_{\mathbb R}^\times$ with $\overline D+\widehat{(s)}\geq 0$, then 
\[
\Supp(\Psi_{\overline D,v}^{\widehat{\mu}_{\max}^{\mathrm{asy}}})\cap \{ x \in X_v^{\an} \mid | s |_{g_v} < 1 \} = \varnothing
\]
for any $v\in M_K\cup K(\mathbb C)$.
\end{Corollary}

We conclude the article by the following question.
\begin{Question}\label{Que:refined}
Let $\overline D$ be an adelic arithmetic $\mathbb R$-Cartier divisor such that $D$ is pseudo-effective and that $\widehat{\mu}_{\max}^{\mathrm{asy}}(\overline D)=0$. Assume that, for any place $v\in M_K\cup K(\mathbb C)$, the union of all algebraic curves lying in $\Supp(\Psi_{\overline D,v}^{\widehat{\mu}_{\max}^{\mathrm{asy}}})$ is contained in the augmented base locus of $D_v$, does the Dirichlet property always hold for $\overline D$? 
\end{Question}


\begin{thebibliography}{99}
\bibitem{Be}
V. G. Berkovich,
{\it Spectral theory and analytic geometry over non-Archimedean fields},
Mathematical surveys and monographs, No. {\bf 33}, AMS, (1990).

\bibitem{BGS94} J.-B. Bost, H. Gillet and C. Soul\'e, {\it Heights of projective varieties and positive Green forms}, Journal of the American Mathematical Society {\bf 7} (1994), no. 4, 903-1027.

\bibitem{Bost_Kunnemann} J.-B. Bost and K. K\"unnemann, {\it Hermitian vector bundles and extension groups on arithmetic schemes. I. Geometry of numbers}, Advances in Mathematics {\bf 223} (2010), no. 3, P. 987-1106. 

\bibitem{Boucksom_Chen} S. Boucksom and H. Chen, {\it Okounkov bodies of filtered linear series}, Compositio Mathematica {\bf 147} (2011), no.4, 1205-1229.


\bibitem{BMPS}
J. I. Burgos i Gil, A. Moriwaki, P. Philippon and M. Sombra,
Arithmetic positivity on toric varieties,
to appear in J. of Alg. Geom.,
(see also arXiv:1210.7692v1 [math.AG]).

\bibitem{CL06}
A. Chambert-Loir, {\it Mesures et \'equidistribution sur des espaces de Berkovich}, J. Reine Angew. Math. {\bf 595} (2006), 215-235. 

\bibitem{CL_Thuillier}
A. Chambert-Loir and A. Thuillier, {\it Mesures de Mahler et \'equidistribution logarithmique}, Annales de l'Institut Fourier {\bf 59} (2009), 977-1014.


\bibitem{Chen10b}
H. Chen, {\it Convergence des polygones de Harder-Narasimhan}, M\'emoire de la Soci\'et\'e Math\'ematique de France {\bf 120} (2010), 1-120.

\bibitem{Chen10}
H. Chen, {\it Arithmetic Fujita approximation}, Annales de l'ENS {\bf 43} (2010), no.4, 555-578.


\bibitem{Chen11}
H. Chen, {\it Differentiability of the arithmetic volume function}, Journal of the London Mathematical Society {\bf 84} (2011), no.2, 365-384.

\bibitem{ELMNP09}
L. Ein, R. Lazarsfeld, M. Musta\c{t}\v{a}, M. Nakamaye and M. Popa, \emph{Restricted volumes and base loci of linear series}, American Journal of Mathematics {\bf 131} (2009), no.3, 607-651.

\bibitem{Gaudron14}
\'E. Gaudron, {\it Minorations simultan\'ees de formes lin\'eaires de logarithmes de nombres alg\'ebrique} (2014), to appear in Bulletin de la SMF.

\bibitem{Gubler98}
W. Gubler, {\it Local heights of subvarieties over non-archimedean fields}, J. reine angew. Math. {\bf 498} (1998), 61-113.

\bibitem{Kaveh_Khovanski12}
K. Kaveh and A. Khovanskii, {\it Algebraic equations and convex bodies}, in Perspectives in analysis, geometry and topology, Progr. Math., vol. 296, Birkh\"auser/Springer, New York (2012), 263-282.

\bibitem{MK}
M. Klimek,
Pluripotential Theory,
London Mathematical Society Monographs, New Series 6, Oxford Science Publications, (1991).

\bibitem{LM09} 
R. Lazarsfeld and M. Musta\c{t}\v{a}, {\it Convex bodies associated to linear series}, Ann. de l'ENS {\bf 42} (2009), no. 5, 783-835.

\bibitem{Maillot00}
V. Maillot, {\it G\'eom\'etrie d'Arkaelov des vari\'et\'ets toriques et fibr\'es en droites int\'egrables}, M\'emoire de la Soci\'et\'e Math\'ematique de France {\bf 80} (2000), vi+129pp.

\bibitem{MoArZariski}
A. Moriwaki,
Zariski decompositions on arithmetic surfaces, 
Publ. Res. Inst. Math. Sci. {\bf 48} (2012), 799-898.

\bibitem{MoD}
A. Moriwaki,
Toward Dirichlet's unit theorem on arithmetic varieties,
Kyoto J. of Math., {\bf 53} (2013), 197--259.

\bibitem{MoAdel}
A. Moriwaki,
Adelic divisors on arithmetic varieties,
preprint (arXiv:1302.1922 [math.AG]).

\bibitem{MoGD}
A. Moriwaki,
\emph{Toward a geometric analogue of Dirichlet's unit theorem},
preprint (arXiv:1311.6307 [math.AG]).

\bibitem{HolDyn}
S. Morosawa, Y. Nishimura, M. Taniguchi and T. Ueda,
Holomorphic Dynamics,
Cambridge Studies in Advanced Mathematics, vol. 66,
Cambridge University Press, 2000.


\bibitem{Yuan07} X. Yuan, \emph{Big line bundles over arithmetic varieties}, Inventiones Mathematicae {\bf 173} (2007), no. 3, p. 603-649.

\bibitem{ZhSmall}
S. Zhang, 
\emph{Small points and adelic metrics}, 
J. Algebraic Geom. {\bf 4} (1995), 281--300.
\end{thebibliography}
\end{document}